\documentclass[12pt]{amsart}
\usepackage{amsmath,amssymb,latexsym, amsthm, amscd, mathrsfs, stmaryrd}  
\usepackage{enumitem}

\usepackage[all]{xy}
\usepackage{hyperref}
\usepackage{color}
\hypersetup{colorlinks,linkcolor=blue,urlcolor=cyan,citecolor=blue}

\newcommand{\nc}{\newcommand}
\nc{\browntext}[1]{\textcolor{brown}{#1}}
\nc{\greentext}[1]{\textcolor{green}{#1}}
\nc{\redtext}[1]{\textcolor{red}{#1}}
\nc{\bluetext}[1]{\textcolor{blue}{#1}}
\nc{\brown}[1]{\browntext{ #1}}
\nc{\green}[1]{\greentext{ #1}}
\nc{\red}[1]{\redtext{ #1}}
\nc{\blue}[1]{\bluetext{ #1}}
\nc{\zb}[1]{\redtext{From zb: #1}}

\newtheorem{thm}{Theorem}  [section]
\newtheorem{cor}[thm]{Corollary}
\newtheorem{lem}[thm]{Lemma}
\newtheorem{prop}[thm]{Proposition}

\newtheorem{definition}[thm]{Definition}
\newtheorem{example}[thm]{Example}

\theoremstyle{remark}
\newtheorem{rem}[thm]{Remark}

\numberwithin{equation}{section}
 
\newcommand{\cbinom}[2]{\left\{ \begin{matrix} #1\\#2 \end{matrix} \right\}}

\newcommand{\mbf}{\mathbf}

\newcommand{\mrm}{\mathrm}
\newcommand{\A}{\mathcal A}
\newcommand{\dvev}[1]{{\mathfrak{t}}_{\ev}^{{(#1)}}}
\newcommand{\dv}[1]{{\mathfrak{t}}_{\odd}^{{(#1)}}}
\newcommand{\dvd}[1]{t_{\odd}^{{(#1)}}}
\newcommand{\dvp}[1]{t_{\ev}^{{(#1)}}}
\newcommand{\ev}{\mrm{ev}}

\newcommand{\p}{\mathfrak p}
\newcommand{\g}{\mathfrak g}

\newcommand{\kk}{h}
\newcommand{\ka}{\kappa}

\newcommand{\la}{\lambda}
\newcommand{\LR}[2]{\left\llbracket \begin{matrix} #1\\#2 \end{matrix} \right\rrbracket}
\newcommand{\m}{\diamondsuit}
\newcommand{\mc}{\mathcal}
\newcommand{\mf}{\mathfrak}
\newcommand{\ms}{\mathscr}
\newcommand{\N}{\mathbb N}
\newcommand{\odd}{\mrm{odd}}
\newcommand{\one}{\mathbf 1}
\newcommand{\ov}{\overline}
\newcommand{\qbinom}[2]{\begin{bmatrix} #1\\#2 \end{bmatrix} }
\newcommand{\Q}{\mathbb Q}
\newcommand{\sll}{\mathfrak{sl}}

\newcommand{\U}{\mbf U}
\newcommand{\Udot}{\dot{\mbf U}}

\newcommand{\UAdot}{{}_\A{\dot{\mbf U}}}
\newcommand{\Ui}{{\mbf U}^\imath}

\newcommand{\vs}{\varsigma}

\newcommand{\vev}{v^+_{2\la} }
\newcommand{\vodd}{v^+_{2\la+1} }
\newcommand{\Y}{\check{E}}
\newcommand{\Z}{\mathbb Z}

\title[Formulae of $\imath$-divided powers, II]{Formulae of $\imath$-divided powers  in ${\bf U}_q(\mathfrak{sl}_2)$, II}

\author[Weiqiang Wang]{Weiqiang Wang}
\author[Collin Berman]{Collin Berman}
\address{ Department of Mathematics\\ University of Virginia\\ Charlottesville, VA 22904}
\email{ww9c@virginia.edu (Wang), cmb5nh@virginia.edu (Berman)}


\begin{document}

 \begin{abstract} 
 The coideal subalgebra of the quantum $\mathfrak{sl}_2$ is a polynomial algebra in a generator $t$ which depends on a parameter $\kappa$. The existence of the $\imath$-canonical basis (also known as the $\imath$-divided powers) for the coideal subalgebra of the quantum $\mathfrak{sl}_2$ were established by Bao and Wang. We establish closed formulae for the $\imath$-divided powers as polynomials in $t$ and also in terms of Chevalley generators of the quantum $\mathfrak{sl}_2$ when the parameter $\kappa$ is an arbitrary $q$-integer. The formulae were known earlier when $\kappa=0, 1$. 
\end{abstract}

\maketitle

\begin{quote}
{\em Dedicated to  Vyjayanthi Chari for her 60th birthday with admiration
}
\end{quote}

\setcounter{tocdepth}{1}
\tableofcontents

\section{Introduction}

%

A highlight in the theory of Drinfeld-Jimbo quantum groups is the construction of canonical bases by Lusztig and Kashiwara; cf.  \cite{Lu93}. Let $(\U, \Ui)$ be a quantum symmetric pair of finite type \cite{Le99} (also cf. \cite{BK19}).  A general theory of ($\imath$-)canonical bases arising from quantum symmetric pairs of finite type was developed recently by Bao and the first author in \cite{BW18b}, where $\imath$-canonical bases on based $\U$-modules and on the modified form of $\Ui$ were constructed.

One notable feature of the definition of quantum symmetric pairs is its dependence on parameters; see \cite{Le99, BK19, BW18b} where various conditions on parameters are imposed at different levels of generalities for various constructions. In \cite{BWW18} it is shown that the unequal parameters of type B Hecke algebras correspond under the $\imath$-Schur duality to certain specializations of parameters in the type AIII/AIV quantum symmetric pairs. 

%

\vspace{3mm}

For the remainder of the paper we let $\U$ be the quantum group of $\sll_2$ over $\Q(q)$ with generators $E, F, K^{\pm 1}$. Let $\Ui$ be the $\Q(q)$-subalgebra of $\U$ (with a parameter $\ka$), which is a polynomial algebra in  $t$, where 
\[
t = F+ q^{-1}EK^{-1}  +\ka K^{-1}. 
\]
Then $\Ui$ is a coideal subalgebra, and $(\U, \Ui)$ is an example of quantum symmetric pairs; cf. \cite{Ko93}. For the consideration of $\imath$-canonical bases (which will be referred to as $\imath$-divided powers from now on) the parameter $\ka$ is taken to be a bar invariant element in $\A =\Z[q,q^{-1}]$; cf. \cite{BW18b}. 

In contrast to the usual quantum $\sll_2$ case where the divided powers are simple to define, finding explicit formulae for the $\imath$-divided powers is a highly nontrivial problem. Closed formulae for the $\imath$-divided powers in $\Ui$ were known earlier only in two distinguished cases when $\ka=0$ or $1$ (cf. \cite{BeW18}); this verified a conjecture in \cite{BW18a} when $\ka=1$.

The goal of this paper is to establish closed formulae  for the $\imath$-divided powers in $\Ui$  as polynomials in $t$ and also viewed as elements in $\U$ (via the embedding of $\Ui$ to $\U$), when $\ka$ is an arbitrary $q$-integer which is clearly bar invariant. For arbitrary $\ov{\ka} =\ka \in \A$, we are able to present a closed formula only for the {\em second} $\imath$-divided power (note the first $\imath$-divided power is simply $t$ itself).

%

\vspace{3mm}

In contrast to the quantum group setting, there are {\em two} $\A$-forms, ${}_\A \Ui_{\rm ev}$ and ${}_\A \Ui_{\rm odd}$, for $\Ui$ corresponding to the parities $\{\rm ev, \rm odd\}$ of highest weights of finite-dimensional simple $\U$-modules \cite{BW18a, BW18b}. As a very special case of a main theorem in \cite{BW18b}, ${}_\A \Ui_{\rm ev}$ (and respectively, ${}_\A \Ui_{\rm odd}$) admits  $\imath$-canonical bases ($= \imath$-divided powers) for an arbitrary parameter $\overline{\ka} =\ka \in \A$, which are invariant with respect to a bar map (which fixes $t$ and hence is not a restriction of Lusztig's bar map on $\U$ to $\Ui$) and satisfy an asympototic compatibility with the $\imath$-canonical bases on finite-dimensional simple $\U$-modules. 

Computations by hand and by Mathematica have 
led us to make an ansatz for the formulae for the $\imath$-divided powers as polynomials in $t$ when $\ka$ is an arbitrary $q$-integer. In further discussions we need to separate the cases when $\ka$ is an even or odd $q$-integer, and let us restrict ourselves to the case for the $q$-integer $\ka$ being even in the remainder of the Introduction. Our ansatz is that the $\imath$-divided powers $\dvev{n}$ in ${}_\A \Ui_{\rm ev}$ and $\dv{n}$ in ${}_\A \Ui_{\rm odd}$ are given as follows: for $a\in \N$, 
\begin{align}
\label{def:idp:evevKa}
\dvev{n} = 
\begin{cases}
\frac{t}{[2a]!}  (t  - [-2a+2])(t -[-2a+4]) \cdots (t -[2a-4]) (t  - [2a-2]), & \text{if } n=2a, \\
\\
\frac{1}{[2a+1]!} (t  - [-2a])(t -[-2a+2]) \cdots (t -[2a-2]) (t  - [2a]), &\text{if } n=2a+1.
\end{cases}
\end{align}
\begin{align}
\label{def:dv:evoddK}
{\small \dv{n} = 
\begin{cases}
\frac{1}{[2a]!}   (t  - [-2a+1] ) (t  - [-2a+3]) \cdots (t  - [2a-3]) (t -[2a-1]), & \text{if } n=2a, \\
\\
\frac{t}{[2a+1]!}     (t  - [-2a+1] ) (t  - [-2a+3]) \cdots (t  - [2a-3]) (t -[2a-1]), &\text{if } n=2a+1.
\end{cases}
}
\end{align}
That is, the $\imath$-divided power formulas as polynomials in $t$ for $\ka$ being even $q$-integers are the same as for $\ka=0$  \cite{BW18a, BeW18}. 

To verify the above formulae \eqref{def:idp:evevKa}-\eqref{def:dv:evoddK} are indeed the $\imath$-canonical bases as defined and established in \cite{BW18b}, we need to verify 2 properties: (1) these polynomials in $t$ lie in the corresponding $\A$-forms of $\Ui$; (2) they satisfy the asymptotic compatibility with $\imath$-canonical bases on finite-dimensional simple $\U$-modules. 

To that end, we find the expansions for the polynomials in $t$ defined in \eqref{def:idp:evevKa}-\eqref{def:dv:evoddK}  in terms of $E, F, K^{\pm 1}$ of $\U$ (via the embedding $\Ui \to \U$); they are given by explicit triple sum formulae. Once these explicit formulae are found, they are verified by lengthy inductions. In the formulation of the expansion formulae, a sequence of degree $n$ polynomials $p^{(n)} (x) =p_n(x)/ [n]!$ which are defined recursively arise naturally; see \S \ref{subsec:pn}. The presence of $p^{(n)} (x)$ makes the formulation of the expansion formula and its proof much more difficult than the distinguished cases when $\ka=0$ or $1$ treated in \cite{BeW18}. 

The polynomials  $p^{(n)} (x)$ satisfy a crucial integrality property in the sense that $p^{(n)} (\ka)\in \A$ which leads to the integral property (1) above. To prove such an integrality we express these polynomials in terms of another sequence of polynomials (which are a reincarnation of the $\imath$-divided powers viewed as polynomials) with some $q^2$-binomial coefficients. Note that $p_0(0)=1$ and $p_n(0)=0$ for $n\ge 1$, our triple sum expansion formula reduces at $\ka=0$ to a double sum formula in \cite{BeW18}. 

When applying the expansion formulas for the $\imath$-divided powers to the highest weight vectors of the finite-dimensional simple $\U$-modules, we establish the asymptotic compatibility property (2) above explicitly. Note this form of compatibility cannot be made as strong as the compatibility between $\imath$-divided powers on $\Ui$ and simple $\U$-modules when $\ka=0$ or $1$ (as expected in \cite{BW18b}).

With Huanchen Bao's help, we compute the closed formulae for the second $\imath$-divided power for an arbitrary parameter $\ov{\ka} =\ka \in \A$; see Appendix~\ref{sec:genK}. It is an interesting open problem to find closed formulae for higher $\imath$-divided power with such an arbitrary parameter $\ka$. 
 
 %
\vspace{3mm}
The paper is organized as follows. There are 4 sections depending on the parities of the weights and of the parameter $\ka$. 

In Section~\ref{sec:evevK}, we recall the basics of the quantum symmetric pair $(\U, \Ui)$. Throughout the section we take $\ka$ to be an arbitrary even $q$-integer. We establish a key integral property of the polynomials  $p^{(n)} (x)$, and use it to formulate and establish the expansion formula for the $\imath$-divided powers $\dvev{n}$ in ${}_\A \Ui_{\rm ev}$. We prove that $\{\dvev{n} | n\ge 0 \}$ form an $\imath$-canonical basis for ${}_\A \Ui_{\rm ev}$ by showing $\dvev{n} \vev$ is an $\imath$-canonical basis element on the finite-dimensional simple $\U$-modules $L(2\la)$, for integers $\la \gg n$. 

In Section~\ref{sec:oddevK}, letting $\ka$ be an arbitrary even $q$-integer, we establish the expansion formula for the $\imath$-divided powers $\dv{n}$ in ${}_\A \Ui_{\rm odd}$. We prove that $\{\dv{n} | n \ge 0\}$ form an $\imath$-canonical basis for ${}_\A \Ui_{\rm odd}$ by showing $\dv{n}\vodd$ is an $\imath$-canonical basis element on the finite-dimensional simple $\U$-modules $L(2\la+1)$, for integers $\la \gg n$. 

In Section~\ref{sec:evoddK}, we take $\ka$ to be an arbitrary odd $q$-integer. We establish a key integral property of another sequence of polynomials  $\p^{(n)} (x)$, and use it to establish the expansion formula for the $\imath$-divided powers $\dvp{n}$ in ${}_\A \Ui_{\rm ev}$. We show that $\{\dvp{n} | n \ge 0\}$ form an $\imath$-canonical basis for ${}_\A \Ui_{\rm ev}$ and $\dvp{n}\vev$ is an $\imath$-canonical basis element on the finite-dimensional simple $\U$-modules $L(2\la)$, for integers $\la \gg n$.

In Section~\ref{sec:oddoddK}, we take $\ka$ to be an arbitrary odd $q$-integer. We establish the expansion formula for the $\imath$-divided powers $\dvd{n}$ in ${}_\A \Ui_{\rm odd}$. We show that $\{\dvd{n} | n \ge 0\}$ form an $\imath$-canonical basis for ${}_\A \Ui_{\rm odd}$ and $\dvd{n}\vodd$ is an $\imath$-canonical basis element on the finite-dimensional simple $\U$-modules $L(2\la+1)$, for integers $\la \gg n$.

In Appendix~\ref{sec:genK}, for arbitrary $\overline{\ka} =\ka \in \A$, we present closed formulae for the second $\imath$-divided powers in ${}_\A \Ui_{\rm ev}$ and ${}_\A \Ui_{\rm odd}$.

\vspace{.3cm}

{\bf Acknowledgement.} 
WW thanks Huanchen Bao for his insightful collaboration. The formula in the Appendix for the second $\imath$-divided power with arbitrary parameter $\ka$ (which was obtained with help from Huanchen) was crucial to this project, and to a large extent this paper grows by exploring for what values for the parameter $\ka$ reasonable formulae for higher divided powers can be obtained. The research of WW  and the undergraduate research of CB are partially supported by a grant from National Science Foundation. Mathematica was used intensively in this work.

\section{The $\imath$-divided powers $\dvev{n}$ for even weights and even $\ka$}
  \label{sec:evevK}

\subsection{The quantum $\sll_2$}

Recall  the quantum group $\U=\U_q(\mathfrak{sl}_2)$ is the $\Q(q)$-algebra generated by $F, E, K, K^{-1}$, subject to the relations:  
\[
KK^{-1}=K^{-1}K=1,
 \;
 EF -FE =\frac{K-K^{-1}}{q-q^{-1}}, 
 \;
 K E =q^2 E K, 
 \;
 K F=q^{-2} FK.
\]
There is an anti-involution $\vs$ of the $\Q$-algebra $\U$:
\begin{align}
 \label{eq:vs}
\vs: \U \longrightarrow \U,
\qquad E\mapsto E, \quad F\mapsto F, \quad K\mapsto K, \quad q\mapsto q^{-1}.
\end{align}

 Let $\A =\Z[q,q^{-1}]$ and $\ka \in \A$.  Set
\begin{equation}  \label{eq:Y}
\Y :=q^{-1}EK^{-1}, 
\qquad
\kk :=\frac{K^{-2}-1}{q^2-1}.
\end{equation}

  Define, for $a\in \Z, n\ge 0$, 
\begin{equation}  \label{kbinom}
\qbinom{\kk;a}{n} =\prod_{i=1}^n \frac{q^{4a+4i-4} K^{-2} -1}{q^{4i} -1},
\qquad 
[\kk;a]= \qbinom{\kk;a}{1}.
\end{equation}
Then we have, for $a\in \Z, n\in \N$, 
\begin{align}
  \label{FEk}
  \begin{split}
F \Y  - & q^{-2} \Y  F =\kk,
\\
\qbinom{\kk;a}{n} F = F \qbinom{\kk;a+1}{n}, &
\qquad \qbinom{\kk;a}{n} \Y  =\Y  \qbinom{\kk;a-1}{n}.
\end{split}
\end{align}
It follows by definition that
\begin{equation}  \label{kbinom2}
\qbinom{\kk;a}{n} =\qbinom{\kk;1+a}{n}
 -q^{4a}K^{-2} \qbinom{\kk;1+a}{n-1}.
\end{equation}

The following formula holds for $n\ge 0$:  
\begin{align}   \label{FYn}
F \Y^{(n)} &=q^{-2n} \Y^{(n)} F +  \Y^{(n-1)} \frac{q^{3-3n} K^{-2} - q^{1-n} }{q^2-1}.
\end{align} 

%


Let 
\[
\Y^{(n)} =\Y^n/[n]!, \quad
F^{(n)} =F^n/[n]!,\quad \text{ for } n\ge 1.
\] 
Then $\Y^{(n)}= q^{-n^2} E^{(n)} K^{-n}$. It is understood that $\Y^{(n)}=0$ for $n<0$ and $\Y^{(0)}=1$.

%
%
\subsection{The coideal subalgebra $\Ui$}

Recall $\Y$ from \eqref{eq:Y}. Let 
\begin{align}
  \label{eq:t}
t &=F + \Y +\ka K^{-1}. 
\end{align}
Denote by $\Ui$ the $\Q(q)$-subalgebra of $\U$ generated by $t$. Then $\Ui$ is a coideal subalgebra of $\U$ and $(\U, \Ui)$ forms a quantum symmetric pair \cite{Ko93, Le99}; also cf. \cite{BW18b}.

Denote by $\Udot$ the modified quantum group of $\sll_2$ \cite{Lu93}, which is the $\Q(q)$-algebra generated by $E\one_\la, F\one_\la$ and the  idempotents $\one_\la$, for $\la\in \Z$. Let $\UAdot$ (respectively, $\UAdot_{\rm{ev}}$, $\UAdot_{\text{odd}}$) be the $\A$-subalgebra of $\Udot$  generated by $E^{(n)} \one_\la, F^{(n)} \one_\la, \one_\la$, for all $n\ge 0$ and for $\la\in \Z$ (respectively, for $\la$ even, for $\la$ odd). Note $\UAdot =\UAdot_{\rm{ev}} \oplus \UAdot_{\text{odd}}$.
There is a natural left action of $\U$ (and hence $\Ui$) on $\Udot$ such that $K \one_\la =q^\la \one_\la$. 
For $\mu\in \N$, denote by $L(\mu)$ the finite-dimensional simple $\U$-module of highest weight $\mu$, and denote by $L(\mu)_\A$ its Lusztig $\A$-form. 

Following \cite{BW18b}, we introduce the following $\A$-subalgebras of $\Ui$, for ${\rm p}\in \{ {\rm ev}, {\rm odd} \}$: 
\begin{align*}
{}_\A \Ui_{\rm p}
&=\{x\in \Ui~|~x u \in \UAdot_{\rm p}, \forall u\in \UAdot_{\rm p} \}
\\
&=\big \{x\in \Ui~|~x v \in L(\mu)_\A, \forall v\in L(\mu)_\A, \forall {\rm p}\equiv\mu\pmod{2} \big \}. 
\end{align*}
By \cite{BW18b}, ${}_\A \Ui_{\rm p}$ is a free $\A$-submodule of $\Ui$ such that $\Ui =\Q(q) \otimes_\A {}_\A \Ui_{\rm p}$, for ${\rm p}\in \{ {\rm ev}, {\rm odd} \}$. 

\subsection{Definition of $\dvev{n}$ for even $\ka$}

In this section~\ref{sec:evevK} we shall always take $\ka$ to be an even $q$-integer, i.e., 
\begin{align}  
   \label{eq:evenK}
\ka & =[2\ell], \quad  \text{ for } \ell  \in \Z.
\end{align}
We shall take the following as a definition of the $\imath$-divided powers $\dvev{n}$. 

\begin{definition}
Set $\dvev{1} =t = F +\Y +\ka K^{-1}$. The divided powers $\dvev{n}$, for $n\ge 1$, are defined by the recursive relations: 
\begin{align}   
  \label{eq:tt}
\begin{split}
t \cdot \dvev{2a-1} &=[2a] \dvev{2a},
\\
t \cdot \dvev{2a} &=  [2a+1] \dvev{2a+1} +   [2a] \dvev{2a-1}, \quad \text{ for } a\ge 1.
\end{split}
\end{align}   
Equivalently, $\dvev{n}$ is defined by the closed formula  \eqref{def:idp:evevKa}. 
\end{definition}
The bar involution $\psi_\imath$ on $\Ui$, which fixed $t$ and sends $q\mapsto q^{-1}$, clearly fixes the above $\imath$-divided powers. 

The following is a variant of \cite[Lemma~2.2]{BeW18} (where $\ka=0$) with the same proof.

\begin{lem} 
  \label{lem:anti}
The anti-involution $\vs$ on $\U$ 
sends $
F\mapsto F,  \Y \mapsto \Y,  K^{-1} \mapsto K^{-1},  q \mapsto q^{-1}$;  in addition, $\vs$ sends 
\[
\kk\mapsto -q^{2} \kk, 
\quad
\dvev{n} \mapsto \dvev{n}, 
\quad
\qbinom{\kk;a}{n}\mapsto (-1)^n q^{2n(n+1)} \qbinom{\kk;1-a-n}{n}, \; \forall a\in\Z, n\in\N.
\] 
\end{lem}

\subsection{Polynomials $p_n(x)$ and $p^{(n)}(x)$}
 \label{subsec:pn}
 
\begin{definition}
For $n\in \N$, the monic polynomial $p_n(x)$ in $x$ of degree $n$ is defined as
\begin{align}
 \label{def:pn}
p_{n+1} =x p_n  + q^{1-2n} [n] [n-1] p_{n-1}, \qquad  p_0 =1.
\end{align}
\end{definition}
Also set $p_n =0$ for all $n<0.$
Note that $p_n$ is an odd polynomial for $n$ odd while it is an even polynomial for $n$ even. These polynomials $p_n$ will appear in the expansion formula for the $\imath$-divided powers in $\U$. 

\begin{example}
Here are $p_n(x)$, for the first few $n\ge 0$:
\begin{align*}
p_0& =1, \qquad p_1 =x,
\qquad 
p_2=x^2,
\qquad
p_3=x^3 +(q^{-4} +q^{-2})x,
\\
p_4&=x^4 + (q^{-4} +q^{-2})(q^{-4} +q^{-2}+2)x^2,
\\
p_5&=x^5 + (q^{-4} +q^{-2}) (q^{-8}+q^{-6}+3q^{-4} +2q^{-2}+3)x^3
\\
& \qquad \qquad + (q^{-4} +q^{-2})^2 (q^{-8}+q^{-6}+2q^{-4} +q^{-2}+1)x.
\end{align*}
\end{example}

Introduce the monic polynomial of degree $n$:
\begin{align*}
g_n(x) =
\begin{cases}
\prod_{i=0}^{m-1} (x^2-[2i]^2), & \text{ if } n=2m \text{ is even};
\\
x \prod_{i=1}^{m} (x^2-[2i]^2), & \text{ if } n=2m+1 \text{ is odd}.
\end{cases}
\end{align*}
Define 
\begin{equation}  \label{eq:divpg}
p^{(n)}(x) =p_n(x)/[n]!, \qquad
g^{(n)}(x) =g_n(x)/[n]!. 
\end{equation} 
The recursion for $p_n$ can be rewritten as
\begin{equation}   \label{eq:f(n)}
[n+1]p^{(n+1)} =x p^{(n)}  + q^{1-2n} [n-1] p^{(n-1)}.
\end{equation}

Our next goal is to show that $p^{(n)}([2\ell]) \in \A$ for all $\ell \in \Z$; cf. Proposition~\ref{prop:fg:ev}. This is carried out by relating $p^{(n)}$ to $g^{(n)}$ for various $n$, and showing that $g^{(n)}([2\ell]) \in \A$ for all $\ell \in \Z$.

\begin{lem} \label{lem:ev:g-integral}
For $\ka =[2\ell]$ with $\ell \in \Z$ (see \eqref{eq:evenK}), we have $g^{(n)}(\ka) \in \A$, for all $n\ge 0$. 
\end{lem}

\begin{proof} 
We separate the cases for $n=2m+1$ and $n=2m$. 
Noting that
\begin{equation}  \label{eq:A}
\ka -[2i] =[2\ell]-[2i] = [\ell-i] (q^{\ell+i} +q^{-\ell-i}),
\end{equation} 
we have
 \begin{align*}
 g^{(2m+1)}(\ka) &= \frac1{[2m+1]!} \prod_{i=-m}^{m} (\ka-[2i])
 = 
 \qbinom{\ell+m}{2m+1} \prod_{i=-m}^{m}(q^{\ell+i} +q^{-\ell-i} ) \in \A.
 \end{align*}

Similarly, we have
 \begin{align*}
 g^{(2m)}(\ka) &= \frac1{[2m]!} \ka \prod_{i=1-m}^{m-1} (\ka-[2i])
 \\
 &= \frac1{[2m]!} \prod_{i=1-m}^{m} (\ka-[2i]) + \frac1{[2m-1]!}  \prod_{i=1-m}^{m-1} (\ka-[2i])
 \\
 &=\qbinom{\ell+m-1}{2m} \prod_{i=1-m}^{m}(q^{\ell+i} +q^{-\ell-i} ) 
  + g^{(2m-1)}(\ka)  \in \A.
  \end{align*}
  The lemma is proved. 
\end{proof}

\begin{prop}
 \label{prop:fg:ev}
 For $m \in \Z_{\ge 1}$, we have
\begin{align*}
p^{(2m)} &= \sum_{a=0}^{m-1} q^{(1-2m)a} \qbinom{m-1}{a}_{q^2} g^{(2m-2a)},
\\
p^{(2m-1)} &= \sum_{a=0}^{m-1} q^{(3-2m)a} \qbinom{m-1}{a}_{q^2} g^{(2m-2a-1)}.
\end{align*}
In particular, we have $p^{(n)} (\ka) \in \A$, for all $\ka =[2\ell]$ with $\ell \in \Z$ and all $n\ge 0$. 
\end{prop}

\begin{proof}  
 It follows from  these formulae for $p^{(n)}$ and Lemma~\ref{lem:ev:g-integral} that $p^{(n)}([2\ell]) \in \A$.

It remains to prove these formulae. Recall
\begin{align}  \label{recursive:f}
 x \cdot p^{(n)} = [n+1] p^{(n+1)}  -  q^{1-2n} [n-1] p^{(n-1)}, \quad \text{ for } n \ge 1.
\end{align}
Also recall
\begin{align}  \label{recursive:g}
\begin{split}
x \cdot g^{(2a-1)} &=[2a] g^{(2a)},
\\
x \cdot g^{(2a)} &=  [2a+1] g^{(2a+1)} +   [2a] g^{(2a-1)}, \quad \text{ for } a\ge 1.
\end{split}
\end{align}

We prove the formulae by induction on $m$, with the base cases for $p^{(1)}$ and $p^{(2)}$ being clear. We separate in two cases. 

(1). Let us prove the formula for $p^{(2m+1)}$, assuming the formulae for $p^{(2m-1)}$ and $p^{(2m)}$.
By \eqref{recursive:f}--\eqref{recursive:g} we have
\begin{align*}
[2m+1]p^{(2m+1)} 
&= x p^{(2m)} +q^{1-4m} [2m-1] p^{(2m-1)}  \\
&= \sum_{a=0}^{m-1}q^{(1-2m)a} \qbinom{m-1}{a}_{q^2} x \cdot g^{(2m-2a)} \\
&\qquad
  +\sum_{a=0}^{m-1} q^{(3-2m)a+1-4m} [2m-1] \qbinom{m-1}{a}_{q^2} g^{(2m-2a-1)}
 \\
&= \sum_{a=0}^{m-1}q^{(1-2m)a} \qbinom{m-1}{a}_{q^2}
      \left( [2m-2a+1]g^{(2m-2a+1)} +[2m-2a]g^{(2m-2a-1)}
    \right) \\
&\qquad
  +\sum_{a=0}^{m-1} q^{(3-2m)a+1-4m} [2m-1] \qbinom{m-1}{a}_{q^2} g^{(2m-2a-1)}.
\end{align*}
By combining the like terms for $g^{(2m-2a-1)}$ above and using the identity
\[
q^{(1-2m)a} [2m-2a] +q^{(3-2m)a+1-4m} [2m-1] =q^{(1-2m)(a+1)} [4m-2a-1], 
\]
we obtain
\begin{align*}
[2m+1]p^{(2m+1)} 
&= \sum_{a=0}^{m-1}q^{(1-2m)a}  [2m-2a+1] \qbinom{m-1}{a}_{q^2}  g^{(2m-2a+1)} 
     \\
&\qquad
  +\sum_{a=0}^{m-1} q^{(1-2m)(a+1)} [4m-2a-1] \qbinom{m-1}{a}_{q^2} g^{(2m-2a-1)}
  \\ %
&= [2m+1]\sum_{a=0}^{m}q^{(1-2m)a} \qbinom{m}{a}_{q^2} g^{(2m-2a+1)}, 
\end{align*}
where the last equation results from shifting the second summation index from $a\to a-1$ and using
the identity
\[
[4m-2a+1] \qbinom{m-1}{a}_{q^2} +q^{(1-2m)a} [4m-2a+1] \qbinom{m-1}{a-1}_{q^2}
=[2m+1] \qbinom{m}{a}_{q^2}. 
\]

(2) We shall prove the formula for $p^{(2m+1)}$, assuming the formulae for $p^{(2m-1)}$ and $p^{(2m)}$.
By \eqref{recursive:f}--\eqref{recursive:g} we have
\begin{align*}
&[2m+2]p^{(2m+2)} 
\\
&= x p^{(2m+1)} +q^{-1-4m} [2m] p^{(2m)}  \\
&= \sum_{a=0}^{m}q^{(1-2m)a} \qbinom{m}{a}_{q^2} x \cdot g^{(2m-2a+1)}
 +\sum_{a=0}^{m-1} q^{(1-2m)a-1-4m} [2m] \qbinom{m-1}{a}_{q^2} g^{(2m-2a)}
 \\
&= \sum_{a=0}^{m} q^{(1-2m)a} [2m-2a+2] \qbinom{m}{a}_{q^2}  g^{(2m-2a+2)}
 +\sum_{a=0}^{m-1} q^{(1-2m)a-1-4m} [2m] \qbinom{m-1}{a}_{q^2} g^{(2m-2a)}
 \\
 &= [2m+2] \sum_{a=0}^{m} q^{(-1-2m)a} \qbinom{m}{a}_{q^2} g^{(2m-2a+2)},
 \end{align*}
where the last equation results from shifting the second summation index from $a\to a-1$ and 
 using the following identity
 \[
 [2m-2a+2] \qbinom{m}{a}_{q^2} 
+ q^{-2-2m} [2m] \qbinom{m-1}{a-1}_{q^2}
=  q^{-2a} [2m+2] \qbinom{m}{a}_{q^2}. 
 \]
The proposition is proved.
\end{proof}

\begin{rem}
\label{rem:positive-p}
Note that $p_n(x) \in \N [q^{-1}] [x]$ for all $n$. Therefore, for all non-negative even $q$-integer $\ka$, we have $p_n (\ka) \in \N [q,q^{-1}]$. 
\end{rem}

\subsection{Formulae for $\dvev{n}$ with even $\ka$}

Recall $p^{(n)}(x)$ from \eqref{def:pn}--\eqref{eq:divpg}. 
\begin{thm}   
 \label{thm:dvev:evKappa}
Assume $\ka$ is an even $q$-integer as in \eqref{eq:evenK}. Then we have, for $m\ge 1$,  
\begin{align}
\dvev{2m} &= 
\sum_{b=0}^{2m} \sum_{a=0}^{b} \sum_{c \geq 0} q^{\binom{2c}{2} -b(2m-b-2c) -a(b-a)} 
p^{(2m-b-2c)}(\kappa) 
\label{t2m:evev2} \\
&\qquad \qquad \qquad\quad \cdot \Y^{(a)}  \qbinom{\kk;1-m}{c} K^{b-2m+2c} F^{(b-a)}, 
\notag
\\
\dvev{2m-1} &= 
\sum_{b=0}^{2m-1} \sum_{a=0}^{b} \sum_{c \geq 0} q^{\binom{2c}{2}+2c -b(2m-b-2c-1) -a(b-a)}
 p^{(2m-b-2c-1)}(\kappa) 
\label{t2m-1:evev2}
\\
&\qquad\qquad\qquad\quad \cdot \Y^{(a)}  \qbinom{\kk;1-m}{c} K^{b-2m+2c+1} F^{(b-a)}. 
\notag
\end{align}
\end{thm}
The proof of Theorem~\ref{thm:dvev:evKappa} will be given in \S\ref{subsec:proof:dvev:evKappa} below. 
Let us list some equivalent formulae here. 
By applying the anti-involution $\vs$ we convert the formulae in Theorem~\ref{thm:dvev:evKappa} into the following.

\begin{thm}   
 \label{thm:dvev:evKappa2}
Assume $\ka$ is an even $q$-integer as in \eqref{eq:evenK}. Then we have, for $m\ge 1$,   
\begin{align}
\dvev{2m} &= \sum_{b=0}^{2m} \sum_{a=0}^{b} \sum_{c \geq 0} (-1)^c q^{3c+b(2m-b-2c) +a(b-a)}  p^{(2m-b-2c)}(\ka) 
\label{t2m:evevFhE} 
\\
&\qquad\qquad\qquad\quad \cdot  F^{(b-a)} K^{b-2m+2c} \qbinom{\kk;m-c}{c}  \Y^{(a)},
\notag
\\ %
\dvev{2m-1} &= 
\sum_{b=0}^{2m-1} \sum_{a=0}^{b} \sum_{c \geq 0} (-1)^c q^{c +b(2m-b-2c-1) +a(b-a)}
 p^{(2m-b-2c-1)}(\kappa) 
\label{t2m-1:evevFhE}
\\
&\qquad\qquad\qquad\quad \cdot  F^{(b-a)} K^{b-2m+2c+1}  \qbinom{\kk;m-c}{c}  \Y^{(a)}.
\notag
\end{align}
\end{thm}

\begin{proof}
By Lemma~\ref{lem:anti}, the anti-involution $\vs$ sends 
$\qbinom{\kk;1-m}{c}\mapsto (-1)^c q^{2c(c+1)} \qbinom{\kk;m-c}{c}$, 
$q\mapsto q^{-1}$, while fixing $\ka, K, \Y^{(a)}, F^{(a)}$ and $\dvev{n}$. The formulae \eqref{t2m:evevFhE}--\eqref{t2m-1:evevFhE} now follows by applying $\vs$ to the formulae  \eqref{t2m:evev2}--\eqref{t2m-1:evev2} in Theorem~\ref{thm:dvev:evKappa}. 
%
\end{proof}
For $c\in \N, m\in\Z$, we define
\begin{align}   \label{cbinom}
  \cbinom{m}{c} &= \prod_{i=1}^c \frac{q^{4(m+i-1)} -1}{q^{-4i}-1}  \in \A,
  \qquad \cbinom{m}{0}=1.
\end{align}
 Note that $\cbinom{m}{c}$ are $q^2$-binomial coefficients up to some $q$-powers.
Let $\la, m \in \Z$ with $m\ge 1$.  We note that
\begin{align}
  \label{eq:qbinom-hc}
\qbinom{\kk;m-c}{c} \one_{2\la} 
&= \prod_{i=1}^c \frac{q^{4(m-\la-c+i-1)} -1}{q^{4i} -1}  \one_{2\la}
= (-1)^c q^{-2c(c+1)} \cbinom{m-\la-c}{c}  \one_{2\la}.
\end{align}

The following corollary is immediate from \eqref{eq:qbinom-hc}, Proposition~\ref{prop:fg:ev}, and Theorem~\ref{thm:dvev:evKappa}. 
\begin{cor}  
We have $\dvev{n} \in {}_\A \Ui_{\rm{ev}}$, for all $n$. 
\end{cor}

\begin{rem}   
Note $p_n(0)=0$ for $n>0$, and $p_0(0)=1$. 
The formulae in Theorems~\ref{thm:dvev:evKappa} and  \ref{thm:dvev:evKappa2} in the special case for $\ka=0$ recover the formulae in \cite[Theorem~2.5, Proposition~2.7]{BeW18}.
\end{rem}

For $n\in \N$, we denote
\begin{align}  \label{eq:divB}
b^{(n)}  = \sum_{a=0}^n  q^{-a(n-a)}  \Y^{(a)}F^{(n-a)}.
\end{align}

\begin{example}
The formulae for $\dvev{n}$ in Theorem~\ref{thm:dvev:evKappa}, for $1\le n \le 3$, read as follows:
\begin{align*}
\dvev{1} 
&= F +\Y + \ka K^{-1},
\\
\dvev{2} 
&= \Y^{(2)} +q^{-1} \Y F + F^{(2)}  + q [\kk;0] 
 + \ka (q^{-1}K^{-1}F + q^{-1} \Y K^{-1})  + \ka^2 \frac{K^{-2}}{[2]},
\\
\dvev{3}  &   = b^{(3)} + q^3[\kk;-1]F+ q^3 \Y [\kk;-1]
\\
&\qquad 
+ \big(q^{-2} \Y^{(2)}K^{-1} +q^{-3} \Y K^{-1}F + q^{-2} K^{-1}F^{(2)} 
+ q^3 [h;-1] K^{-1} \big) \ka
\\
&\qquad + \frac{q^{-2}}{[2]} (\Y K^{-2} +K^{-2}F) \ka^2
+\frac{\ka^3 + (q^{-4} +q^{-2})\ka}{[3]!} K^{-3}. 
\end{align*}
\end{example}

\subsection{The $\imath$-canonical basis for $\Udot^\imath_{\rm{ev}}$ for even $\ka$}
  \label{sec:iCB:ev}
  
Denote by $L(\mu)$ the  
simple $\U$-module of highest weight $\mu \in \N$, with highest weight vector $v^+_\mu$.
Then $L(\mu)$ admits a canonical basis $\{ F^{(a)} v^+_\mu\mid 0\le a \le \mu\}$. 
Following \cite{BW18a, BW18b}, there exists a new bar involution $\psi_\imath$ on $L(\mu)$, which, in our current rank one setting, can be defined simply by requiring $\dvev{n} v^+_\mu$ to be $\psi_\imath$-invariant for all $n$. 
As a very special case of the general results in \cite[Corollary~F]{BW18b} (also cf. \cite{BW18a}), we know that
 the $\imath$-canonical basis $\{b^{\mu}_{a} \}_{0\le a \le \mu}$ of $L(\mu)$ exists and is characterized by the following 2 properties: 

(iCB1) $b^{\mu}_{a}$ is $\psi_\imath$-invariant; \qquad 

(iCB2) $b^{\mu}_{a} \in F^{(a)} v^+_\mu + \sum_{0\le r< a} q^{-1} \Z[q^{-1}] F^{(r)} v^+_\mu$.

 \begin{thm}  \label{thm:iCB:ev}
   $\quad$
\begin{enumerate}
\item
Let $n\in \N$. 
For each integer $\la \gg n$, the element $\dvev{n} \vev$ is an $\imath$-canonical basis element for $L(2\la)$. 
 
 \item
The set $\{\dvev{n} \mid n \in \N \}$ forms the $\imath$-canonical basis for ${\U}^\imath$ (and an $\A$-basis in ${}_\A \Ui_{\rm{ev}}$). 
\end{enumerate}
 \end{thm}
 
 \begin{proof}
Let $\la, m \in \N$ with $m\ge 1$.  
We compute by Theorem~\ref{thm:dvev:evKappa2}(1) and \eqref{eq:qbinom-hc} that 
\begin{align} 
\dvev{2m} \vev 
&= \sum_{b=0}^{2m} \sum_{c \geq 0} (-1)^c q^{3c+b(2m-b-2c)}  p^{(2m-b-2c)}(\ka) 
F^{(b)} K^{b-2m+2c} \qbinom{\kk;m-c}{c} \vev 
\notag \\
&= \sum_{b=0}^{2m} \sum_{c \geq 0} q^{(b-2\la)(2m-b-2c)-2c^2+c}  p^{(2m-b-2c)}(\ka) 
\cbinom{m-\la-c}{c}F^{(b)}   \vev 
  \label{dvev:2m-mod}.
\end{align}
Similarly we compute by Theorem~\ref{thm:dvev:evKappa2}(2) that 
\begin{align} 
\dvev{2m-1} \vev 
&= 
\sum_{b=0}^{2m-1} \sum_{c \geq 0} (-1)^c q^{c +b(2m-b-2c-1)}
 p^{(2m-b-2c-1)}(\kappa) 
F^{(b)} K^{b-2m+2c+1}  \qbinom{\kk;m-c}{c}  \vev
\notag
\\
&= 
\sum_{b=0}^{2m-1} \sum_{c \geq 0}  q^{(b-2\la)(2m-b-2c-1)-2c^2-c}
 p^{(2m-b-2c-1)}(\kappa) 
\cbinom{m-\la-c}{c} F^{(b)}  \vev. 
 \label{dvev:2m-1-mod}
\end{align}

Note that
\begin{equation}
  \label{eq:q-1}
\cbinom{m-\la-c}{c} \in \N[q^{-1}], \qquad \text{ for }\la \ge m, c\ge 0.
\end{equation}  

Let $n\in \N$. Recall $\ka =[2\ell]$, for $\ell \in \Z$.  
One computes that $\deg_q p^{(n)} (\ka) =(2\ell-1)n - {n \choose 2}$ and hence $q^{(b-2\la)(n-b-2c)-2c^2 \pm c}
 p^{(n-b-2c)}(\kappa) \in q^{-1} \N[q^{-1}]$ when $\la \gg n>b+2c$. It follows by \eqref{dvev:2m-mod}, \eqref{dvev:2m-1-mod} and \eqref{eq:q-1} that 
 \begin{align}
   \label{eq:lattice}
 \dvev{n} \vev  & \in F^{(n)}  \vev + \sum_{b<n}q^{-1} \Z[q^{-1}] F^{(b)} \vev,
 \qquad \text{ for } \la \gg n.
 \end{align}
 Therefore, the first statement follows by the characterization properties (iCB1)--(iCB2) of the $\imath$-canonical basis for $L(2\la)$.
 
 The second statement follows now from the definition of  the $\imath$-canonical basis for ${\U}^\imath$ using
the projective system $\{ L(2\la) \}_{\la\in \N}$ with $\Ui$-homomorphisms $L(2\la+2) \rightarrow L(2\la)$; cf. \cite[\S6]{BW18b}; this set $\{\dvev{n} \mid n \in \N \}$ forms an $\A$-basis for  ${}_\A \Ui_{\rm{ev}}$. 
The theorem is proved. 
 \end{proof}
 
 \begin{rem}
It follows by \eqref{dvev:2m-mod} that 
\[
\dvev{2m} \vev   \in F^{(2m)}  \vev +  q^{2m-2\la-1} \ka F^{(2m-1)}  \vev +\sum_{k\ge 2} \A \cdot F^{(2m-k)}  \vev.
\]
Recall $\ka=[2\ell]$. When $\ell -1 \ge \la -m \ge 0$, we have $q^{2m-2\la-1} \ka \not \in q^{-1} \Z[q^{-1}]$, and therefore, the nonzero element $\dvev{2m} \vev$ is not an $\imath$-canonical basis element in $L(2\la)$. So we cannot expect to have a stronger form of Theorem~\ref{thm:iCB:ev} as \cite[Theorem 2.10(1)]{BeW18} in  the case when $\ka=0$ (i.e., $\ell=0$). 
 \end{rem}


%
%
\subsection{Proof of Theorem~\ref{thm:dvev:evKappa} }
\label{subsec:proof:dvev:evKappa}
We prove the formulae for $\dvev{n}$ by induction on $n$, in two steps (1)--(2) below.  The base cases when $n=1,2$ are clear. 

(1) We shall prove the formula \eqref{t2m:evev2} for $\dvev{2m}$, assuming the formula \eqref{t2m-1:evev2} for $\dvev{2m-1}$.

Recall $[2m] \dvev{2m} = t \cdot \dvev{2m-1}$, and $t= F+ \Y  +\ka K^{-1}$. 
Let us compute 
\[
I=\Y  \dvev{2m-1},
\qquad
II=F \dvev{2m-1},
\qquad
III=\ka K^{-1} \dvev{2m-1}.
\]
First we have
\begin{align*}
I 
&=\sum_{b=0}^{2m-1} \sum_{a=0}^{b} \sum_{c \geq 0} q^{\binom{2c}{2}+2c -b(2m-b-2c-1) -a(b-a)}
    [a+1] p^{(2m-b-2c-1)}(\kappa) \\
  & \qquad\qquad \cdot \Y^{(a+1)}  \qbinom{\kk;1-m}{c} K^{b-2m+2c+1} F^{(b-a)}
\\
&=\sum_{b=0}^{2m} \sum_{a=0}^{b} \sum_{c \geq 0} q^{\binom{2c}{2}+2c -(b-1)(2m-b-2c) -(a-1)(b-a)}
    [a] p^{(2m-b-2c)}(\kappa) \\
  & \qquad\qquad \cdot \Y^{(a)}  \qbinom{\kk;1-m}{c} K^{b-2m+2c} F^{(b-a)},
\end{align*}
where the last equation is obtained by shifting indices $a\to a-1$, $b\to b-1$ (and then adding some zero terms to make the bounds of the summations uniform throughout).
Using \eqref{FYn} we  have
\begin{align*}
II  
&=\sum_{b=0}^{2m-1} \sum_{a=0}^{b} \sum_{c \geq 0} q^{\binom{2c}{2}+2c -b(2m-b-2c-1) -a(b-a)}
   p^{(2m-b-2c-1)}(\kappa) \\
  & \qquad \cdot 
  \left( 
   q^{-2a} \Y^{(a)} F + \Y^{(a-1)} \frac{q^{3-3a}K^{-2}-q^{1-a}}{q^2-1}
  \right )
   \qbinom{\kk;1-m}{c} K^{b-2m+2c+1} F^{(b-a)} =II^1 +II^2,
\end{align*}
where $II^1$ and $II^2$ are the two natural summands associated to the plus sign.
By shifting index $b\to b-1$ and then adding some zero terms, we further have
\begin{align*}
II^1&=\sum_{b=0}^{2m} \sum_{a=0}^{b} \sum_{c \geq 0} q^{\binom{2c}{2}+2c - (b-1)(2m-b-2c) -a(b-a-1)-2a+2b+4c-4m}
   [b-a] p^{(2m-b-2c)}(\kappa) \\
  & \qquad \qquad \cdot  \Y^{(a)}  \qbinom{\kk;-m}{c} K^{b-2m+2c} F^{(b-a)}.
\end{align*}
By shifting the indices $a\to a+1$, $b\to b+1$ and $c\to c-1$, we further have
\begin{align*}
II^2
 &=\sum_{b=0}^{2m} \sum_{a=0}^{b} \sum_{c \geq 0} q^{\binom{2c-2}{2}+2c - (b+1)(2m-b-2c) - (a+1)(b-a)-2}
  p^{(2m-b-2c)}(\kappa) \\
  & \qquad \qquad \cdot  \Y^{(a)} \frac{q^{-3a}K^{-2}-q^{-a}}{q^2-1} \qbinom{\kk;1-m}{c-1} K^{b-2m+2c} F^{(b-a)}.
\end{align*}
Using \eqref{recursive:f} we also compute
\begin{align*}
III 
&=\sum_{b=0}^{2m} \sum_{a=0}^{b} \sum_{c \geq 0} q^{\binom{2c}{2}+2c -b(2m-b-2c-1) -a(b-a) -2a}
    \ka \cdot p^{(2m-b-2c-1)}(\kappa) \\
  & \qquad\qquad \cdot \Y^{(a)}  \qbinom{\kk;1-m}{c} K^{b-2m+2c} F^{(b-a)}
\\
&=\sum_{b=0}^{2m} \sum_{a=0}^{b} \sum_{c \geq 0} q^{\binom{2c}{2}+2c -b(2m-b-2c-1) -a(b-a) -2a}
    [2m-b-2c] p^{(2m-b-2c)}(\kappa)  \\
  & \qquad\qquad \cdot \Y^{(a)}  \qbinom{\kk;1-m}{c} K^{b-2m+2c} F^{(b-a)}
  \\
 & \quad+\sum_{b=0}^{2m} \sum_{a=0}^{b} \sum_{c \geq 0} q^{\binom{2c}{2}+2c -b(2m-b-2c-1) -a(b-a) -2a -4m}
    [2m-b-2c] p^{(2m-b-2c)}(\kappa)  \\
  & \qquad\qquad \cdot \Y^{(a)}  \qbinom{\kk;1-m}{c-1} K^{b-2m+2c-2} F^{(b-a)}.
\end{align*}
(Note that we have shifted the index $c\to c-1$ in the last summand above.)

Collecting the formulae for $I, II^1, II^2$ and $III$  gives us
\[
t \cdot \dvev{2m-1} = 
 \sum_{0 \le a \le b \le 2m}  \sum_{c\ge 0}  p^{(2m-b-2c)}(\kappa)  \Y^{(a)} H_{a,b,c}  K^{b-2m+2c} F^{(b-a)}, 
\]
where   
\begin{align*}
H_{a,b,c} & :=
 q^{\binom{2c}{2}+2c -(b-1)(2m-b-2c) -(a-1)(b-a)} [a]  \qbinom{\kk;1-m}{c} 
  \\
& + q^{\binom{2c}{2}+2c - (b-1)(2m-b-2c) -a(b-a-1)-2a+2b+4c-4m} [b-a]  \qbinom{\kk;-m}{c} 
\\
& + q^{\binom{2c}{2}-2c+1 - (b+1)(2m-b-2c) - (a+1)(b-a)}
  \frac{q^{-3a}K^{-2}-q^{-a}}{q^2-1} \qbinom{\kk;1-m}{c-1}    
\\
&+ q^{\binom{2c}{2}+2c -b(2m-b-2c-1) -a(b-a) -2a} [2m-b-2c] \qbinom{\kk;1-m}{c}  
  \\
& - q^{\binom{2c}{2}+2c -b(2m-b-2c-1) -a(b-a) -2a -4m} [2m-b-2c] \qbinom{\kk;1-m}{c-1} K^{-2}.
\end{align*}

Recall $[2m] \dvev{2m} = t \cdot \dvev{2m-1}$. To prove the formula \eqref{t2m:evev2},  by the PBW basis theorem it suffices to prove the following identity, for all $a,b,c$:
\begin{align}  \label{eq:H}
H_{a,b,c} = q^{\binom{2c}{2} -b(2m-b-2c) -a(b-a)} [2m]  \qbinom{\kk;1-m}{c}.
\end{align}
Thanks to $q^{2m-a} [a] -[2m] =q^{-a} [a-2m]$, we can combine the RHS with the first summand of LHS \eqref{eq:H}.
Hence, after canceling out the $q$-powers $q^{\binom{2c}{2} -b(2m-b-2c) -a(b-a)-a}$ on both sides, we see that \eqref{eq:H} is equivalent to the following identity, for all $a, b, c$:
\begin{equation}  \label{ABCD}
A+B+C+D_1 +D_2=0,
\end{equation}
where   
\begin{align*}
A &= [a-2m]  \qbinom{\kk;1-m}{c},
  \\
B &=    q^{4c-2m+b} [b-a] \qbinom{\kk;-m}{c},
\\
C   &=  q^{2a-2m+1} \frac{q^{-3a}K^{-2}-q^{-a}}{q^2-1} \qbinom{\kk;1-m}{c-1},
\\
D_1 &=  q^{2c+b-a} [2m-b-2c] \qbinom{\kk; 1-m}{c},
   \\
D_2 &= -  q^{2c-4m+b-a} [2m-b-2c] \qbinom{\kk; 1-m}{c-1} K^{-2}.
\end{align*}

Let us prove the identity \eqref{ABCD}. Using \eqref{kbinom2}, we can write $B=B_1+B_2$, where 
\begin{align*}
B_1 &=  q^{4c-2m+b} [b-a]  \qbinom{\kk; 1-m}{c},
   \quad
B_2 = -  q^{4c-6m+b} [b-a] \qbinom{\kk; 1-m}{c-1} K^{-2}.
\end{align*}
Noting that  
\[
\frac{q^{-3a}K^{-2}-q^{-a}}{q^2-1} = q^{4m-4c-3a}\frac{( q^{4c-4m}K^{-2}-1)}{q^2-1}
 +  q^{2m-2c-2a-1} [2m-2c-a],
 \] 
we rewrite $C=C_1+C_2$, where
\begin{align*}
C_1 &= q^{-2c+2m-a} [2c] \qbinom{\kk; 1-m}{c},
   \quad
C_2 =   q^{-2c} [2m-2c-a]  \qbinom{\kk; 1-m}{c-1}.
\end{align*}

A direct computation gives us
\begin{align*}
A+B_1+C_1+D_1 
&= (q-q^{-1})  [2c] [2m-2c-a]  \qbinom{\kk; 1-m}{c},
\\
C_2 +(B_2+D_2)  
&= C_2 -  q^{2c-4m} [2m-2c-a] \qbinom{\kk; 1-m}{c-1} K^{-2}
\\
&= - (q-q^{-1}) [2c] [2m-2c-a] \qbinom{\kk; 1-m}{c}.
\end{align*}
Summing up these two equations, we have $A+B+C+D_1+D_2=0,$ whence \eqref{ABCD}, completing Step~(1). 

\vspace{3mm}

(2) Assuming the formulae for $\dvev{n}$ with $n\le 2m$, we shall now prove the following formula for $\dvev{2m+1}$ (obtained from \eqref{t2m-1:evev2} with $m$ replaced by $m+1$):
\begin{align}
\dvev{2m+1} &= 
\sum_{b=0}^{2m+1} \sum_{a=0}^{b} \sum_{c \geq 0} q^{\binom{2c}{2}+2c -b(2m-b-2c+1) -a(b-a)}
 p^{(2m-b-2c+1)}(\kappa) 
\label{t2m+1:evev2}
\\
&\qquad\qquad\qquad\quad \cdot \Y^{(a)}  \qbinom{\kk;-m}{c} K^{b-2m+2c-1} F^{(b-a)}. 
\notag
\end{align}
The proof is based on the recursion $t \cdot \dvev{2m} =[2m+1] \dvev{2m+1} +[2m] \dvev{2m-1}$. 
Recall $t= F +\Y +\ka K^{-1}$ and $\dvev{2m}$ from \eqref{t2m:evev2}. We shall compute 
\[
\texttt I=\Y  \dvev{2m},
\qquad
\texttt{II}=F \dvev{2m},
\qquad
\texttt{III} =\ka K^{-1} \dvev{2m}, 
\]
respectively. First we have
\begin{align*}
\texttt{I} 
&= 
\sum_{b=0}^{2m} \sum_{a=0}^{b} \sum_{c \geq 0} q^{\binom{2c}{2} -b(2m-b-2c) -a(b-a)} 
 [a+1] p^{(2m-b-2c)}(\kappa)  
 \Y^{(a+1)}  \qbinom{\kk;1-m}{c} K^{b-2m+2c} F^{(b-a)}\\
&= 
 \sum_{b=0}^{2m+1} \sum_{a=0}^{b} \sum_{c \geq 0} q^{\binom{2c}{2} -(b-1)(2m-b-2c+1) -(a-1)(b-a)} 
[a] p^{(2m-b-2c+1)}(\kappa) \\
&\qquad\qquad\qquad\qquad 
\cdot \Y^{(a)}  \qbinom{\kk;1-m}{c} K^{b-2m+2c-1} F^{(b-a)},
\end{align*}
where the last equation is obtained by shifting indices $a\to a-1$, $b\to b-1$. We also have
\begin{align*}
\texttt{II} 
&= \sum_{b=0}^{2m} \sum_{a=0}^{b} \sum_{c \geq 0} q^{\binom{2c}{2} -b(2m-b-2c) -a(b-a)} 
p^{(2m-b-2c)}(\kappa) 
\\
& \qquad \cdot 
  \left( 
   q^{-2a} \Y^{(a)} F + \Y^{(a-1)} \frac{q^{3-3a}K^{-2}-q^{1-a}}{q^2-1}
  \right )
   \qbinom{\kk;1-m}{c} K^{b-2m+2c} F^{(b-a)} 
=\texttt{II}_1 +\texttt{II}_2,
\end{align*}
where $\texttt{II}_1$ and $\texttt{II}_2$ are the two natural summands associated to the plus sign.
By shifting the index $b\to b-1$, we have
\begin{align*}
\texttt{II}_1&=\sum_{b=0}^{2m+1} \sum_{a=0}^{b} \sum_{c \geq 0} q^{\binom{2c}{2} -(b-1)(2m-b-2c+1) -a(b-a-1)-2a +2(b+2c-2m-1)} 
[b-a] p^{(2m-b-2c+1)}(\kappa) 
\\
&\qquad\qquad \qquad \cdot 
  \Y^{(a)}   
      \qbinom{\kk;-m}{c} K^{b-2m+2c-1} F^{(b-a)}.
\end{align*}
By shifting the indices $a\to a+1$, $b\to b+1$ and $c\to c-1$, we further have
\begin{align*}
\texttt{II}_2
&=\sum_{b=0}^{2m+1} \sum_{a=0}^{b} \sum_{c \geq 0} q^{\binom{2c-2}{2} -(b+1)(2m-b-2c+1) -(a+1)(b-a)} 
p^{(2m-b-2c+1)}(\kappa) 
\\
& \qquad\qquad \qquad \cdot 
  \Y^{(a)} \frac{q^{-3a}K^{-2}-q^{-a}}{q^2-1}
   \qbinom{\kk;1-m}{c-1} K^{b-2m+2c-1} F^{(b-a)}.
\end{align*}
Using \eqref{recursive:f} we also compute
\begin{align*}
\texttt{III}  
%
&= \sum_{b=0}^{2m+1} \sum_{a=0}^{b} \sum_{c \geq 0} q^{\binom{2c}{2} -b(2m-b-2c) -a(b-a) -2a} 
[2m-b-2c+1]  
\\
&\qquad\qquad\qquad\qquad \cdot p^{(2m-b-2c+1)}(\kappa) \Y^{(a)}  \qbinom{\kk;1-m}{c} K^{b-2m+2c-1} F^{(b-a)}
\\
& \quad + \sum_{b=0}^{2m+1} \sum_{a=0}^{b} \sum_{c \geq 0} q^{\binom{2c}{2} -b(2m-b-2c) -a(b-a) -2a-4m} 
[2m-b-2c+1]  
\\
&\qquad\qquad\qquad\qquad \cdot p^{(2m-b-2c+1)}(\kappa) \Y^{(a)}  \qbinom{\kk;1-m}{c-1} K^{b-2m+2c-3} F^{(b-a)}.
\end{align*}
(Note that we have shifted the index $c\to c-1$ in the last summand above.)

Collecting the formulae for $\texttt{I}, \texttt{II}_1, \texttt{II}_2$, and $\texttt{III}$, we obtain 
\[
t \cdot \dvev{2m} = 
 \sum_{0 \le a \le b \le 2m+1}  \sum_{c\ge 0} p^{(2m-b-2c+1)}(\kappa)  \Y^{(a)} L_{a,b,c} K^{b-2m+2c-1} F^{(b-a)},
\]
where  
\begin{align*}
L_{a,b,c} := &
q^{\binom{2c}{2} -(b-1)(2m-b-2c+1) -(a-1)(b-a)} [a] \qbinom{\kk;1-m}{c}  
\\
&+ q^{\binom{2c}{2} -(b-1)(2m-b-2c+1) -a(b-a-1)-2a +2(b+2c-2m-1)} [b-a]  \qbinom{\kk;-m}{c}  
\\
&+ q^{\binom{2c}{2}-4c+3 -(b+1)(2m-b-2c+1) -(a+1)(b-a)} 
 \frac{q^{-3a}K^{-2}-q^{-a}}{q^2-1}   \qbinom{\kk;1-m}{c-1}   \\
&+ q^{\binom{2c}{2} -b(2m-b-2c) -a(b-a) -2a} [2m-b-2c+1]  \qbinom{\kk;1-m}{c}  
\\
&+ q^{\binom{2c}{2} -b(2m-b-2c) -a(b-a) -2a-4m} [2m-b-2c+1]  \qbinom{\kk;1-m}{c-1} K^{-2}.
\end{align*}

On the other hand, using \eqref{t2m-1:evev2} (with an index shift $c\to c-1$) and \eqref{t2m+1:evev2} we write 
\begin{align*}
[2m+1] & \dvev{2m+1} +[2m] \dvev{2m-1} 
\\
&=\sum_{0 \le a \le b \le 2m+1} \sum_{c \geq 0} p^{(2m-b-2c+1)}(\kappa) 
\Y^{(a)} R_{a,b,c} K^{b-2m+2c-1} F^{(b-a)}, 
\end{align*}
where  
\begin{align*}
R_{a,b,c} :=&
q^{\binom{2c}{2}+2c -b(2m-b-2c+1) -a(b-a)}   [2m+1] \qbinom{\kk;-m}{c} 
\\
&+ q^{\binom{2c}{2}-2c+1 -b(2m-b-2c+1) -a(b-a)} [2m]  
 \qbinom{\kk;1-m}{c-1}.   
\end{align*}

To prove the formula \eqref{t2m+1:evev2} for $\dvev{2m+1}$, it suffices to show that, for all $a,b,c$,
\begin{equation}
  \label{L=R}
L_{a,b,c}=R_{a,b,c}.
\end{equation}
Canceling the $q$-powers $q^{\binom{2c}{2}+2bc-b(2m-b+1)-a(b-a)}$ on both sides, we see that the identity \eqref{L=R} is equivalent to the following identity, for all $a,b,c$:
\begin{align}
 \label{l=r}
 \begin{split}
q^{2m-2c-a+1} [a] \qbinom{\kk;1-m}{c}  
&+ q^{2c-2m+b-a-1} [b-a] \qbinom{\kk;-m}{c}  
\\
&+ q^{-2m-2c+a+2} \frac{q^{-3a}K^{-2}-q^{-a}}{q^2-1}   \qbinom{\kk;1-m}{c-1}    \\
&+ q^{b -2a} [2m-b-2c+1]  \qbinom{\kk;1-m}{c}  
\\
&+ q^{b -2a -4m} [2m-b-2c+1] \qbinom{\kk;1-m}{c-1} K^{-2}
\\
=& q^{2c} [2m+1] \qbinom{\kk;-m}{c}  
 + q^{-2c+1} [2m] \qbinom{\kk;1-m}{c-1}.  
   \end{split}
\end{align}

By combining the second summand of LHS with the first summand of RHS as well as combining the third summand of LHS with the second summand of RHS, the identity \eqref{l=r} is reduced to the following equivalent identity, for all $a, b, c$:
\begin{equation}  \label{WXYZ}
W+X+Y+Z_1+Z_2=0,
\end{equation}
where  
\begin{align*}
W&=  [a]  \qbinom{\kk;1-m}{c},
\\
X&= q^{4c-2m+b-1} [-2m+b-a-1] \qbinom{\kk;-m}{c},
\\
Y&= \frac{q^{-4m-a+1}K^{-2}-q^{a+1}}{q^2-1}   \qbinom{\kk;1-m}{c-1},  \\
Z_1&=  q^{2c-2m+b -a-1} [2m-b-2c+1] \qbinom{\kk;1-m}{c},
  \\
Z_2&=  q^{2c-6m+b-a-1} [2m-b-2c+1] \qbinom{\kk;1-m}{c-1} K^{-2}.  
\end{align*}

Let us prove the identity \eqref{WXYZ}. Using \eqref{kbinom2}, we can write $X=X_1+X_2$, where 
\begin{align*}
X_1 &=q^{4c-2m+b-1} [-2m+b-a-1] \qbinom{\kk; 1-m}{c},
   \\
X_2 &= - q^{4c-6m+b-1} [-2m+b-a-1] \qbinom{\kk; 1-m}{c-1} K^{-2}.
\end{align*}
Noting that  
\[
\frac{q^{-4m-a+1}K^{-2}-q^{a+1}}{q^2-1}
= q^{-4c-a+1}\frac{( q^{4c-4m}K^{-2}-1)}{q^2-1}
 +  q^{-2c} [-2c-a],
 \] 
we rewrite $Y=Y_1+Y_2$, where
\begin{align*}
Y_1 &=q^{-2c-a} [2c] \qbinom{\kk; 1-m}{c},
   \qquad
Y_2 = q^{-2c} [-2c-a] \qbinom{\kk; 1-m}{c-1}.
\end{align*}
A direct computation shows that
\begin{align*}
W+X_1+Y_1+Z_1 &=
- (q-q^{-1}) [2c+a] [2c]  \qbinom{\kk; 1-m}{c},
\\
(X_2+Z_2) +Y_2 
&= q^{2c-4m} [2c+a] \qbinom{\kk; 1-m}{c-1} K^{-2} +Y_2
\\
&= (q-q^{-1}) [2c+a] [2c] \qbinom{\kk; 1-m}{c}.
\end{align*}
Summing up these two equations, we obtain $W+X+Y+Z_1+Z_2=0$, whence \eqref{WXYZ}, completing Step ~(2).

The proof of Theorem~\ref{thm:dvev:evKappa} is completed. \qed

\section{The $\imath$-divided powers $\dv{n}$ for odd weights and even $\ka$}
  \label{sec:oddevK}

In this section we shall always take $\ka$ to be an even $q$-integer, i.e., 
\[
\ka =[2\ell], \qquad  \text{ for } \ell  \in \Z.
\]  

\subsection{Definition of $\dv{n}$ for even $\ka$}

We introduce some notations. Set, for $n\ge 1, a\in \Z$, 
\begin{equation}   \label{brace}
\LR{\kk;a}{0}=1, 
\qquad
\LR{\kk;a}{n}= \prod_{i=1}^n \frac{q^{4a+4i-4} K^{-2}-q^2}{q^{4i}-1}, 
\qquad \llbracket \kk;a\rrbracket = \LR{\kk;a}{1}.
\end{equation}
It follows from \eqref{brace} that, for $n\ge 0$ and $a\in \Z$, 
\begin{align}
 \label{eq:commLR}
 \begin{split}
\LR{\kk;a}{n} F &= F  \LR{\kk;a+1}{n},
\qquad
\LR{\kk;a}{n} \Y =\Y \LR{\kk;a-1}{n},
\\
\LR{\kk;a}{n} &=\LR{\kk;1+a}{n} -q^{4a}K^{-2} \LR{\kk;1+a}{n-1}.  
  \end{split}
\end{align}

We shall take the following as a definition of the $\imath$-divided powers $\dv{n}$. 

\begin{definition}
Set $\dv{1} =t= F +\Y +\ka K^{-1}$. The divided powers $\dv{n}$, for $n\ge 1$, are defined by the  recursive relations: 
\begin{align}   
  \label{eq:tt:oddevKa}
\begin{split}
t \cdot \dv{2a-1} &=[2a] \dv{2a}+   [2a-1] \dv{2a-2},
\\
t \cdot \dv{2a} &=  [2a+1] \dv{2a+1} , \quad \text{ for } a\ge 1.
\end{split}
\end{align}   
Equivalently, $\dv{n}$ is defined by the closed formula \eqref{def:dv:evoddK}.
\end{definition}

The following lemma is a variant of \cite[Lemma~3.3]{BeW18} (where $\ka=0$) with the same proof.

\begin{lem}   
  \label{lem:anti2}
The anti-involution $\vs$ on $\U$ fixes $F,  \Y,  K,  \dvd{n}$, respectively. Moreover, $\vs$ sends 
$q \mapsto q^{-1}$,  and 
\[
\LR{\kk;a}{n}\mapsto (-1)^n q^{2n(n-1)}  \LR{\kk;2-a-n}{n}, \quad \forall a \in \Z,\; n\in \N. 
\]
\end{lem}

\subsection{Formulae of $\dv{n}$ with even $\ka$}

Recall the polynomials $p^{(n)}$, for $n\ge 0$, from \S\ref{subsec:pn}.

\begin{thm}  
  \label{thm:dv:evenKappa}
Let $\ka$ be an even $q$-integer. Then we have, for $m\ge 0$,  
\begin{align}
\dv{2m} &= 
\sum_{b=0}^{2m}  \sum_{a=0}^{b} \sum_{c \geq 0} q^{\binom{2c}{2} - b(2m-b-2c) -a(b-a)} 
p^{(2m-b-2c)}(\kappa) 
\label{t2m:oddev} \\
&\qquad \qquad \qquad\quad \cdot \Y^{(a)}  \LR{\kk;1-m}{c} K^{b-2m+2c} F^{(b-a)}, 
\notag
\\%
\dv{2m+1} &= 
\sum_{b=0}^{2m+1} \sum_{a=0}^{b} \sum_{c \geq 0} q^{\binom{2c}{2}-2c - b(2m-b-2c+1) -a(b-a)}
 p^{(2m-b-2c+1)}(\kappa) 
\label{t2m+1:oddev}
\\
&\qquad\qquad\qquad\quad \cdot \Y^{(a)}  \LR{\kk;1-m}{c} K^{b-2m+2c-1} F^{(b-a)}. 
\notag
\end{align}
\end{thm}

The proof of Theorem~\ref{thm:dv:evenKappa} will be given in \S\ref{subsec:proof:dv:evKappa} below. 
By applying the anti-involution $\vs$ we convert the formulae in Theorem~\ref{thm:dv:evenKappa} into the following.

\begin{thm}  
  \label{thm:dv:evenKappa2}
Let $\ka$ be an even $q$-integer. Then we have, for $m\ge 0$,  
\begin{align}
\dv{2m} &= 
\sum_{b=0}^{2m}  \sum_{a=0}^{b} \sum_{c \geq 0} (-1)^c q^{-c +b(2m-b-2c) +a(b-a)} 
p^{(2m-b-2c)}(\kappa) 
\label{t2m:oddevFhE}  \\
&\qquad \qquad \qquad\qquad \cdot F^{(b-a)} K^{b-2m+2c} \LR{\kk;1+m-c}{c}  \Y^{(a)}
\notag
\\
\dv{2m+1} &=  
\sum_{b=0}^{2m+1} \sum_{a=0}^{b} \sum_{c \geq 0} (-1)^c q^{c +b(2m-b-2c+1) +a(b-a)}
 p^{(2m-b-2c+1)}(\kappa) 
\label{t2m+1:oddevFhE}  \\
&\qquad\qquad\qquad\qquad \cdot F^{(b-a)} K^{b-2m+2c-1} \LR{\kk;1+m-c}{c}  \Y^{(a)}.
\notag
\end{align}
\end{thm}

\begin{proof}
Recall from Lemma~\ref{lem:anti2} that the anti-involution $\vs$ on $\U$   
fixes $F,  \Y,  K,  \dvd{n}, \ka$ while sending
$q \mapsto q^{-1}$, 
$\LR{\kk;1-m}{c}\mapsto (-1)^c q^{2c(c-1)}  \LR{\kk;1+m-c}{c}$.
The formulae  \eqref{t2m:oddevFhE}--\eqref{t2m+1:oddevFhE} now follows from \eqref{t2m:oddev}--\eqref{t2m+1:oddev}.


\end{proof}

  Let $\la, m \in \Z$ and $c\in \N$.  Recall $\cbinom{m}{c}$ from \eqref{cbinom}.
 We note that
\begin{align}
  \label{eq:LRhc}
\LR{\kk; 1+m-c}{c} \one_{2\la+1}  
&= \prod_{i=1}^c \frac{q^{4m-4c+4i} K^{-2} -q^2}{q^{4i} -1} \one_{2\la+1}   
= (-1)^c q^{-2c^2} \cbinom{m-\la-c}{c} \one_{2\la+1}.
\end{align}

The following corollary is immediate from \eqref{eq:LRhc}, Proposition~\ref{prop:fg:ev}, and Theorem~\ref{thm:dv:evenKappa}. 
\begin{cor}  
We have $\dv{n} \in {}_\A \Ui_{\rm{odd}}$, for all $n$. 
\end{cor}

\begin{example}
The formulae of $\dv{n}$, for $1\le n\le 3$, in Theorem~\ref{thm:dv:evenKappa} reads as follows. 
\begin{align*}
\dv{1} 
&= F +\Y + \ka K^{-1},
\\%
\dv{2} 
&=  b^{(2)}  + q\llbracket \kk;0 \rrbracket + \ka (q^{-1}K^{-1}F + q^{-1} \Y K^{-1})  +  \frac{\ka^{2}}{[2]} K^{-2},  
 \\%
 \dv{3}  &   = b^{(3)} + q^{-1} \llbracket \kk;0 \rrbracket F + q^{-1} \Y \llbracket \kk;0 \rrbracket
\\
 & \qquad +(q^{-2} \Y^{(2)} K^{-1} + q^{-3} \Y K^{-1} F + q^{-2} K^{-1} F^{(2)} + q^{-1} \llbracket \kk;0 \rrbracket K^{-1})\ka
\\
 & \qquad + \frac{q^{-2}}{[2]}(\Y K^{-2} + K^{-2} F)\ka^2
 + \frac{K^{-3}}{[3]!}\big( \ka^3 + (q^{-4} + q^{-2})\ka \big).
\end{align*}
\end{example}

\begin{rem}
The formula for $\dv{2m}$ is formally obtained from $\dvev{2m}$ in Theorem~\ref{thm:dvev:evKappa},
with $\qbinom{h;a}{c}$ replaced by $\LR{h;a}{c}$.
The formula for $\dv{2m-1}$ is formally obtained from $\dvev{2m-1}$ in Theorem~\ref{thm:dvev:evKappa},
with $\qbinom{h;a}{c}$ replaced by $q^{-4c}\LR{h;a+1}{c}$.
\end{rem}

\begin{rem}   
Note $p_n(0)=0$ for $n>0$, and $p_0(0)=1$. 
The formulae in Theorems~\ref{thm:dv:evenKappa} and  \ref{thm:dv:evenKappa2} in the special case for $\ka=0$ recover the formulae in \cite[Theorem~3.1, Proposition~3.4]{BeW18}.
\end{rem}

\subsection{The $\imath$-canonical basis for $\Udot_{\text{odd}}$ for even $\ka$}
  \label{sec:iCB:odd}
  
Recall from \S\ref{sec:iCB:ev} the $\imath$-canonical basis on simple $\U$-modules $L(\mu)$, for $\mu \in \N$. 
  
\begin{thm}  \label{thm:iCB:odd}
    $\quad$
\begin{enumerate}
\item
Let $n\in \N$. 
For each integer $\la \gg n$, the element $\dv{n} \vodd$ is an $\imath$-canonical basis element for $L(2\la+1)$. 
 
 \item
The set $\{\dv{n} \mid n \in \N \}$ forms the $\imath$-canonical basis for ${\U}^\imath$ (and an $\A$-basis in ${}_\A \Ui_{\rm{odd}}$). 
\end{enumerate}
\end{thm}
 
 \begin{proof}
  Recall $\cbinom{m}{c}$ from \eqref{cbinom} and $\LR{\kk; a}{c}$ from \eqref{brace}.  Let $\la, m \in \N$.  
It follows by a direct computation using Theorem~\ref{thm:dv:evenKappa2} and \eqref{eq:LRhc} that 
\begin{align}
\dv{2m} \vodd &= 
\sum_{b=0}^{2m} \sum_{c \geq 0} (-1)^c q^{-c +b(2m-b-2c)} 
p^{(2m-b-2c)}(\kappa) 
\notag
 \\
&\qquad \qquad \qquad\qquad \cdot F^{(b)} K^{b-2m+2c} \LR{\kk;1+m-c}{c}  \vodd
\notag
\\
&=  \sum_{b=0}^{2m} \sum_{c \geq 0} q^{-2c^2-c +(b-2\la-1)(2m-b-2c)}  p^{(2m-b-2c)}(\kappa) \cbinom{m-\la-c}{c}  F^{(b)}  \vodd.
\label{t2m:oddevF}  
\end{align}
Similarly using Theorem~\ref{thm:dv:evenKappa2} we have 
\begin{align}
\dv{2m+1} \vodd &=  
\sum_{b=0}^{2m+1} \sum_{c \geq 0} (-1)^c q^{c +b(2m-b-2c+1)}
 p^{(2m-b-2c+1)}(\kappa) 
\notag  \\
&\qquad\qquad\qquad\qquad \cdot F^{(b)} K^{b-2m+2c-1} \LR{\kk;1+m-c}{c} \vodd
\notag
\\
&=  \sum_{b=0}^{2m+1} \sum_{c \geq 0} q^{-2c^2+c +(b-2\la-1)(2m-b-2c+1)} p^{(2m-b-2c+1)}(\kappa) \cbinom{m-\la-c}{c} F^{(b)} \vodd. 
\label{t2m+1:oddevF} 
\end{align}

By a similar argument for \eqref{eq:lattice}, using \eqref{t2m:oddevF}--\eqref{t2m+1:oddevF}  we obtain
 \begin{align*}
 \dv{n} \vodd  & \in F^{(n)}  \vodd + \sum_{b<n}q^{-1} \Z[q^{-1}] F^{(b)} \vodd,
 \qquad \text{ for } \la \gg n.
 \end{align*}
 
The second statement follows now from the definition of the $\imath$-canonical basis for $\Ui$ using
the projective system $\{ L(2\la+1) \}_{\la\ge 0}$; cf. \cite[\S6]{BW18b}. 
 \end{proof}

\subsection{Proof of Theorem~\ref{thm:dv:evenKappa} }
   \label{subsec:proof:dv:evKappa}

 
We prove the formulae for $\dv{n}$ by induction on $n$, in two steps (1)--(2) below. The base cases when $n=1,2$ are clear. 

(1) We shall prove the formula \eqref{t2m+1:oddev} for $\dv{2m+1}$, assuming the formula \eqref{t2m:oddev} for $\dv{2m}$.

Recall $[2m+1] \dvev{2m+1} = t \cdot \dv{2m}$, and $t= F +\Y +\ka K^{-1}$. 
Let us compute 
\[
I=\Y  \dv{2m},
\qquad
II=F \dv{2m},
\qquad
III=\ka K^{-1} \dv{2m}.
\]

First we have
\begin{align*}
I 
&= \sum_{b=0}^{2m}  \sum_{a=0}^{b} \sum_{c \geq 0} 
q^{\binom{2c}{2} - b(2m-b-2c) -a(b-a)} [a+1] p^{(2m-b-2c)}(\kappa) 
 \\
&\qquad \qquad \qquad\quad \cdot \Y^{(a+1)}  \LR{\kk;1-m}{c} K^{b-2m+2c} F^{(b-a)}
\\%
&= \sum_{b=0}^{2m+1}  \sum_{a=0}^{b} \sum_{c \geq 0} 
q^{\binom{2c}{2} - (b-1)(2m-b-2c+1) -(a-1)(b-a)} [a] p^{(2m-b-2c+1)}(\kappa) 
 \\
&\qquad \qquad \qquad\quad \cdot \Y^{(a)}  \LR{\kk;1-m}{c} K^{b-2m+2c-1} F^{(b-a)},
\end{align*}
where the last equation is obtained by shifting indices $a\to a-1$, $b\to b-1$ (and then adding some zero terms to make the bounds of the summations uniform throughout).
Using \eqref{FYn} we  have
\begin{align*}
II  
&=   \sum_{b=0}^{2m}  \sum_{a=0}^{b} \sum_{c \geq 0} 
q^{\binom{2c}{2} - b(2m-b-2c) -a(b-a)}  p^{(2m-b-2c)}(\kappa) 
 \\
&\qquad \qquad \qquad\quad \cdot 
 \left( 
   q^{-2a} \Y^{(a)} F + \Y^{(a-1)} \frac{q^{3-3a}K^{-2}-q^{1-a}}{q^2-1}
  \right )
    \LR{\kk;1-m}{c} K^{b-2m+2c} F^{(b-a)}
\end{align*}  
where $II^1$ and $II^2$ are the two natural summands associated to the plus sign.
By shifting index $b\to b-1$, we further have
\begin{align*}
II^1
&=\sum_{b=0}^{2m+1}  \sum_{a=0}^{b} \sum_{c \geq 0} 
q^{\binom{2c}{2} - (b+1)(2m-b-2c+1) -a(b-a+1)} [b-a] p^{(2m-b-2c+1)}(\kappa) 
 \\
&\qquad \qquad \qquad\quad \cdot 
 \Y^{(a)} 
       \LR{\kk;-m}{c} K^{b-2m+2c-1} F^{(b-a)}. 
\end{align*}
By shifting the indices $a\to a+1$, $b\to b+1$ and $c\to c-1$, we further have
\begin{align*}
II^2 
&=\sum_{b=0}^{2m+1}  \sum_{a=0}^{b} \sum_{c \geq 0} 
q^{\binom{2c-2}{2} - (b+1)(2m-b-2c+1) -(a+1)(b-a)}  p^{(2m-b-2c+1)}(\kappa) 
 \\
&\qquad \qquad \qquad\quad \cdot 
 \Y^{(a)} \frac{q^{-3a}K^{-2}-q^{-a}}{q^2-1}
    \LR{\kk;1-m}{c-1} K^{b-2m+2c-1} F^{(b-a)}. 
\end{align*}  
Using \eqref{recursive:f}
we also compute
\begin{align*}
III 
&= \sum_{b=0}^{2m}  \sum_{a=0}^{b} \sum_{c \geq 0} 
q^{\binom{2c}{2} - b(2m-b-2c) -a(b-a)-2a} 
\ka \cdot p^{(2m-b-2c)}(\kappa) 
 \\
&\qquad \qquad \qquad\quad \cdot \Y^{(a)}  \LR{\kk;1-m}{c} K^{b-2m+2c-1} F^{(b-a)}
 \\
&= \sum_{b=0}^{2m+1}  \sum_{a=0}^{b} \sum_{c \geq 0} 
q^{\binom{2c}{2} - b(2m-b-2c) -a(b-a)-2a} 
[2m-b-2c+1]  
 \\
&\qquad \qquad \qquad\quad \cdot p^{(2m-b-2c+1)}(\kappa) \Y^{(a)}  \LR{\kk;1-m}{c} K^{b-2m+2c-1} F^{(b-a)}
\\%
&\quad 
- \sum_{b=0}^{2m+1}  \sum_{a=0}^{b} \sum_{c \geq 0} 
q^{\binom{2c-2}{2} - (b+2)(2m-b-2c+2) -a(b-a)-2a+1} 
[2m-b-2c+1] 
 \\
&\qquad \qquad \qquad\quad \cdot p^{(2m-b-2c+1)}(\kappa) \Y^{(a)}  \LR{\kk;1-m}{c-1} K^{b-2m+2c-3} F^{(b-a)}.
\end{align*}
(Note that we have shifted the index $c\to c-1$ in the last summand above.)

Collecting the formulae for $I, II^1, II^2$ and $III$  gives us
\[
t \cdot \dv{2m} = 
 \sum_{0 \le a \le b \le 2m+1}  \sum_{c\ge 0}  p^{(2m-b-2c+1)}(\kappa)  \Y^{(a)} 
 \mc H_{a,b,c}  K^{b-2m+2c-1} F^{(b-a)}, 
\]
where   
\begin{align*}
\mc H_{a,b,c}  :=
& \; q^{\binom{2c}{2} - (b-1)(2m-b-2c+1) -(a-1)(b-a)} [a] \LR{\kk;1-m}{c}  
\\
&+ q^{\binom{2c}{2} - (b+1)(2m-b-2c+1) -a(b-a+1)} [b-a] \LR{\kk;-m}{c}  
\\
&+ q^{\binom{2c-2}{2} - (b+1)(2m-b-2c+1) -(a+1)(b-a)}
 \frac{q^{-3a}K^{-2}-q^{-a}}{q^2-1} \LR{\kk;1-m}{c-1}  
\\
&+ q^{\binom{2c}{2} - b(2m-b-2c) -a(b-a)-2a} 
[2m-b-2c+1]  \LR{\kk;1-m}{c}
\\
&- q^{\binom{2c-2}{2} - (b+2)(2m-b-2c+2) -a(b-a)-2a+1} 
[2m-b-2c+1] \LR{\kk;1-m}{c-1} K^{-2}.
\end{align*}
Recall $[2m+1] \dv{2m+1} = t \cdot \dv{2m}$. To prove the formula \eqref{t2m+1:oddev} for $\dv{2m+1}$,  by the PBW basis theorem and the inductive assumption it suffices to prove the following identity, for all $a,b,c$:
\begin{align}  \label{eq:H3}
\mc H_{a,b,c} = q^{\binom{2c}{2}-2c - b(2m-b-2c+1) -a(b-a)} [2m+1] \LR{\kk;1-m}{c}.
\end{align}
Thanks to $q^{2m-a+1} [a] -[2m+1] = q^{-a} [a-2m-1]$, we can combine the RHS with the first summand of LHS \eqref{eq:H3}.
Hence, after canceling out the $q$-powers $q^{\binom{2c}{2}-2c - b(2m-b-2c+1) -a(b-a) -a}$
on both sides, we see that \eqref{eq:H3} is equivalent to the following identity, for all $a, b, c$:
\begin{equation}  \label{ABCD3}
\mc A+\mc B+\mc C+\mc D_1 +\mc D_2=0,
\end{equation}
where 
\begin{align*}
\mc A &= [a-2m-1] \LR{\kk;1-m}{c},
\\
\mc B &= q^{4c-2m+b-1} [b-a]  \LR{\kk;-m}{c},   
\\
\mc C &= q^{2a-2m+2} \frac{q^{-3a}K^{-2}-q^{-a}}{q^2-1} \LR{\kk;1-m}{c-1},   
\\
\mc D_1 &= q^{2c+b-a} [2m-b-2c+1]  \LR{\kk;1-m}{c},
\\
\mc D_2 &= - q^{2c-4m+b-a} [2m-b-2c+1] \LR{\kk;1-m}{c-1} K^{-2}.
\end{align*}

Let us prove the identity \eqref{ABCD3}. Using \eqref{eq:commLR}, we can write $\mc B=\mc B_1+\mc B_2$, where 
\begin{align*}
\mc B_1 &=  q^{4c-2m+b-1} [b-a]  \LR{\kk;-m}{c},
   \quad
\mc B_2 = -  q^{4c-6m+b-1} [b-a]  \LR{\kk;-m}{c}.
\end{align*}
Noting that  
\[
\frac{q^{-3a}K^{-2}-q^{-a}}{q^2-1} = q^{4m-4c-3a}\frac{( q^{4c-4m}K^{-2}-q^2)}{q^2-1}
 +  q^{2m-2c-2a} [2m-2c-a+1],
 \] 
we rewrite $\mc C=\mc C_1+\mc C_2$, where
\begin{align*}
\mc C_1 &= q^{2m-2c-a+1} [2c] \LR{\kk; 1-m}{c},
   \quad
\mc C_2 =   q^{2-2c} [2m-2c-a+1]  \LR{\kk; 1-m}{c-1}.
\end{align*}

A direct computation gives us
\begin{align*}
\mc A+\mc B_1+\mc C_1+\mc D_1 
&= (q-q^{-1}) [2c] [2m-2c-a+1]  \LR{\kk; 1-m}{c},
\\
\mc C_2 +(\mc B_2+\mc D_2)  
&= \mc C_2 -  q^{2c-4m} [2m-2c-a+1] \LR{\kk; 1-m}{c-1} K^{-2}
\\
&= - (q-q^{-1}) [2c] [2m-2c-a+1]  \LR{\kk; 1-m}{c}.
\end{align*}
Summing up these two equations, we have $\mc A+\mc B+\mc C+\mc D_1+\mc D_2=0,$ whence \eqref{ABCD3}, completing Step~(1). 

\vspace{3mm}
(2) Assuming the formulae for $\dv{n}$ with $n\le 2m+1$, we shall now prove the following formula \eqref{t2m+2:oddev} for $\dv{2m+2}$ (obtained with $m$ replaced by $m+1$ in \eqref{t2m:oddev}):
\begin{align}
\dv{2m+2} &= 
\sum_{b=0}^{2m+2}  \sum_{a=0}^{b} \sum_{c \geq 0} q^{\binom{2c}{2} - b(2m-b-2c+2) -a(b-a)} 
p^{(2m-b-2c+2)}(\kappa) 
  \label{t2m+2:oddev}
 \\
&\qquad \qquad \qquad\quad \cdot \Y^{(a)}  \LR{\kk;-m}{c} K^{b-2m+2c-2} F^{(b-a)}. 
\notag
\end{align}
The proof is based on the recursion $t \cdot \dv{2m+1} =[2m+2] \dv{2m+2} +[2m+1] \dv{2m}$. 
Recall $t= F +\Y +\ka K^{-1}$ and $\dv{2m+1}$ from \eqref{t2m+1:oddev}. We shall compute 
\[
\texttt{I}=\Y  \dv{2m+1},
\qquad
\texttt{II}=F \dv{2m+1},
\qquad
\texttt{III} =\ka K^{-1} \dv{2m+1}, 
\]
respectively. First by by shifting indices $a\to a-1$ and $b\to b-1$, we have
\begin{align*}
\texttt{I} 
&=\sum_{b=0}^{2m+1} \sum_{a=0}^{b} \sum_{c \geq 0} q^{\binom{2c}{2}-2c - b(2m-b-2c+1) -a(b-a)}
[a+1] p^{(2m-b-2c+1)}(\kappa) 
\\
&\qquad\qquad\qquad\quad \cdot \Y^{(a+1)}  \LR{\kk;1-m}{c} K^{b-2m+2c-1} F^{(b-a)}
\\
&=\sum_{b=0}^{2m+2} \sum_{a=0}^{b} \sum_{c \geq 0} 
q^{\binom{2c}{2}-2c - (b-1)(2m-b-2c+2) -(a-1)(b-a)}
[a] p^{(2m-b-2c+2)}(\kappa) 
\\
&\qquad\qquad\qquad\quad \cdot \Y^{(a)}  \LR{\kk;1-m}{c} K^{b-2m+2c-2} F^{(b-a)}.
\end{align*}
Using \eqref{FYn} we   also have
\begin{align*}
\texttt{II} 
&= \sum_{b=0}^{2m+1} \sum_{a=0}^{b} \sum_{c \geq 0} q^{\binom{2c}{2}-2c - b(2m-b-2c+1) -a(b-a)}
 p^{(2m-b-2c+1)}(\kappa) 
\\
&\quad \cdot 
 \left( 
   q^{-2a} \Y^{(a)} F + \Y^{(a-1)} \frac{q^{3-3a}K^{-2}-q^{1-a}}{q^2-1}
  \right )
 \LR{\kk;1-m}{c} K^{b-2m+2c-1} F^{(b-a)}
=\texttt{II}_1 +\texttt{II}_2,
\end{align*}
where $\texttt{II}_1$ and $\texttt{II}_2$ are the two natural summands associated to the plus sign.
By shifting the index $b\to b-1$, we have
\begin{align*}
\texttt{II}_1 
&= \sum_{b=0}^{2m+2} \sum_{a=0}^{b} \sum_{c \geq 0} 
q^{\binom{2c}{2}-2c - (b+1)(2m-b-2c+2) -a(b-a-1)-2a}
[b-a] p^{(2m-b-2c+2)}(\kappa) 
\\
&\qquad\qquad\qquad\quad \cdot 
    \Y^{(a)} 
    \LR{\kk;-m}{c} K^{b-2m+2c-2} F^{(b-a)}. 
\end{align*}
By shifting the indices $a\to a+1$, $b\to b+1$ and $c\to c-1$, we further have
\begin{align*}
\texttt{II}_2 
&= \sum_{b=0}^{2m+2} \sum_{a=0}^{b} \sum_{c \geq 0} 
q^{\binom{2c-2}{2}-2c - (b+1)(2m-b-2c+2) -(a+1)(b-a)+2}
 p^{(2m-b-2c+2)}(\kappa) 
\\
&\qquad\qquad\qquad\quad \cdot 
 \Y^{(a)} \frac{q^{-3a}K^{-2}-q^{-a}}{q^2-1}
 \LR{\kk;1-m}{c-1} K^{b-2m+2c-2} F^{(b-a)}.
 \end{align*}

Using \eqref{eq:pn2div} we also compute
\begin{align*}
\texttt{III}  
&= \sum_{b=0}^{2m+1} \sum_{a=0}^{b} \sum_{c \geq 0} q^{\binom{2c}{2}-2c - b(2m-b-2c+1) -a(b-a)-2a}
\ka \cdot p^{(2m-b-2c+1)}(\kappa) 
\\
&\qquad\qquad\qquad\quad \cdot \Y^{(a)}  \LR{\kk;1-m}{c} K^{b-2m+2c-2} F^{(b-a)}
\\
&= \sum_{b=0}^{2m+2} \sum_{a=0}^{b} \sum_{c \geq 0} q^{\binom{2c}{2}-2c - b(2m-b-2c+1) -a(b-a)-2a}
[2m-b-2c+2]  
\\
&\qquad\qquad\qquad\quad \cdot 
p^{(2m-b-2c+2)}(\kappa) \Y^{(a)}  \LR{\kk;1-m}{c} K^{b-2m+2c-2} F^{(b-a)}
\\%
& - \sum_{b=0}^{2m+2} \sum_{a=0}^{b} \sum_{c \geq 0} 
q^{\binom{2c-2}{2}-2c - (b+2)(2m-b-2c+3) -a(b-a)-2a+3}
[2m-b-2c+2] 
\\
&\qquad\qquad\qquad\quad \cdot 
p^{(2m-b-2c+2)}(\kappa) \Y^{(a)}  \LR{\kk;1-m}{c-1} K^{b-2m+2c-4} F^{(b-a)}.
\end{align*}
(Note that we have shifted the index $c\to c-1$ in the last summand above.)

Collecting the formulae for $\texttt{I}, \texttt{II}_1, \texttt{II}_2$, and $\texttt{III}$, we obtain 
\[
t \cdot \dv{2m+1} = 
 \sum_{0 \le a \le b \le 2m+2}  \sum_{c\ge 0} p^{(2m-b-2c+2)}(\kappa)  
  \Y^{(a)} \mc L_{a,b,c} K^{b-2m+2c-2} F^{(b-a)},
\]
where  
\begin{align*}
\mc L_{a,b,c} := 
& \; q^{\binom{2c}{2}-2c - (b-1)(2m-b-2c+2) -(a-1)(b-a)} 
[a] \LR{\kk;1-m}{c} 
\\
&+q^{\binom{2c}{2}-2c - (b+1)(2m-b-2c+2) -a(b-a-1)-2a} 
[b-a]  \LR{\kk;-m}{c}  
\\
&+q^{\binom{2c-2}{2}-2c - (b+1)(2m-b-2c+2) -(a+1)(b-a)+2}
\frac{q^{-3a}K^{-2}-q^{-a}}{q^2-1} \LR{\kk;1-m}{c-1}  
\\
&+q^{\binom{2c}{2}-2c - b(2m-b-2c+1) -a(b-a)-2a}
[2m-b-2c+2]  \LR{\kk;1-m}{c}
\\
&- q^{\binom{2c-2}{2}-2c - (b+2)(2m-b-2c+3) -a(b-a)-2a+3}
[2m-b-2c+2]  \LR{\kk;1-m}{c-1} K^{-2}. 
\end{align*}

On the other hand, using \eqref{t2m:oddev} (with an index shift $c\to c-1$) and \eqref{t2m+2:oddev} we write 
\begin{align*}
 [2m+2] & \dv{2m+2} +[2m+1] \dv{2m} 
\\
&=\sum_{0 \le a \le b \le 2m+2} \sum_{c \geq 0} 
p^{(2m-b-2c+2)}(\kappa) \Y^{(a)} \mc R_{a,b,c} K^{b-2m+2c-2} F^{(b-a)}, 
\end{align*}
where  
\begin{align*}
\mc R_{a,b,c} 
&:= q^{\binom{2c}{2} - b(2m-b-2c+2) -a(b-a)} 
[2m+2]  \LR{\kk;-m}{c}  
\\
&\qquad + q^{\binom{2c-2}{2} - b(2m-b-2c+2) -a(b-a)} 
[2m+1]  \LR{\kk;1-m}{c-1}.    
\end{align*}

To prove the formula \eqref{t2m+2:oddev} for $\dv{2m+2}$, it suffices to show that, for all $a,b,c$,
\begin{equation}
  \label{L=R3}
\mc L_{a,b,c}= \mc R_{a,b,c}.
\end{equation}
Canceling the $q$-powers $q^{\binom{2c}{2}-2c - (b-1)(2m-b-2c+2) -(a-1)(b-a)}$ on both sides, we see that the identity \eqref{L=R3} is equivalent to the following identity, for all $a,b,c$:
\begin{align}
 \label{l=r3}
 \begin{split}
[a] \LR{\kk;1-m}{c}  
&+q^{4c-4m+b-4}  [b-a]  \LR{\kk;-m}{c}   
\\
&+q^{2a-4m+1} \frac{q^{-3a}K^{-2}-q^{-a}}{q^2-1} \LR{\kk;1-m}{c-1}   
\\
&+q^{2c-2m+b-a-2} [2m-b-2c+2]  \LR{\kk;1-m}{c}
\\
&- q^{2c-6m+b-a-2} [2m-b-2c+2]  \LR{\kk;1-m}{c-1} K^{-2}
\\
&= q^{4c-2m+a-2}  [2m+2]  \LR{\kk;-m}{c}  
+ q^{a-2m+1}  [2m+1]  \LR{\kk;1-m}{c-1}.  
   \end{split}
\end{align}

By combining the second summand of LHS with the first summand of RHS as well as combining the third summand of LHS with the second summand of RHS, the identity \eqref{l=r3} is reduced to the following equivalent identity, for all $a, b, c$:
\begin{equation}  \label{WXYZ3}
\mc W+\mc X+\mc Y+\mc Z_1+\mc Z_2=0,
\end{equation}
where  
\begin{align*}
\mc W & =[a] \LR{\kk;1-m}{c},
\\
\mc X & =q^{4c-2m+b-2}  [b-a-2m-2]  \LR{\kk;-m}{c},   
\\
\mc Y & =  \frac{q^{-a-4m+1}K^{-2}-q^{a+3}}{q^2-1} \LR{\kk;1-m}{c-1},    
\\
\mc Z_1 & =q^{2c-2m+b-a-2} [2m-b-2c+2]  \LR{\kk;1-m}{c},
\\
\mc Z_2 & = - q^{2c-6m+b-a-2} [2m-b-2c+2]  \LR{\kk;1-m}{c-1} K^{-2}.
\end{align*}

Let us finally prove the identity \eqref{WXYZ3}. Using \eqref{eq:commLR}, we can write $\mc X=\mc X_1+\mc X_2$, where 
\begin{align*}
\mc X_1 &=q^{4c-2m+b-2}  [b-a-2m-2]  \LR{\kk; 1-m}{c},
   \\
\mc X_2 &= - q^{4c-6m+b-2}  [b-a-2m-2]  \LR{\kk; 1-m}{c-1} K^{-2}.
\end{align*}
Noting that  
\[
\frac{q^{-a-4m+1}K^{-2}-q^{a+3}}{q^2-1}  
= q^{-4c-a+1}\frac{( q^{4c-4m}K^{-2}-q^2)}{q^2-1}
 +  q^{2-2c} [-2c-a],
 \] 
we rewrite $\mc Y=\mc Y_1+\mc Y_2$, where
\begin{align*}
\mc Y_1 &=q^{-2c-a} [2c] \LR{\kk; 1-m}{c},
   \qquad
\mc Y_2 = q^{2-2c} [-2c-a] \LR{\kk; 1-m}{c-1}.
\end{align*}
A direct computation shows that
\begin{align*}
\mc W+\mc X_1+\mc Y_1+\mc Z_1 &=
- (q-q^{-1}) [2c+a] [2c]  \LR{\kk; 1-m}{c},
\\
(\mc X_2+\mc Z_2) +\mc Y_2 
&= q^{2c-4m} [2c+a] \LR{\kk; 1-m}{c-1} K^{-2} +\mc Y_2
\\
&= (q-q^{-1}) [2c+a] [2c] \LR{\kk; 1-m}{c}.
\end{align*}
Summing up these two equations, we obtain $\mc W+\mc X+\mc Y+\mc Z_1+\mc Z_2=0$, whence \eqref{WXYZ3}, completing Step ~(2).

The proof of Theorem~\ref{thm:dv:evenKappa} is completed. \qed

\section{The $\imath$-divided powers  $\dvp{n}$ for even weights and odd $\ka$}
  \label{sec:evoddK}

In this section~\ref{sec:evoddK} we shall always take $\ka$ to be an odd $q$-integer, i.e., 
\begin{align}  
   \label{eq:oddK}
\ka & =[2\ell-1], \quad  \text{ for } \ell  \in \Z.
\end{align}
\subsection{Definition of $\dvp{n}$ for odd $\ka$}

We shall take the following as a definition of the $\imath$-divided powers  $\dvp{n}$ with odd $\ka$, which is formally identical to the formulae for $\dv{n}$ with even $\ka$. 

\begin{definition}
Set $\dvp{1} = t =  F +\Y +\ka K^{-1}$. The divided powers $\dvp{n}$, for $n\ge 1$, are defined by the recursive relations: 
\begin{align}   
  \label{eq:tt:oddevKa}
\begin{split}
t \cdot \dvp{2a-1} &=[2a] \dvp{2a}+   [2a-1] \dvp{2a-2},
\\
t \cdot \dvp{2a} &=  [2a+1] \dvp{2a+1} , \quad \text{ for } a\ge 1.
\end{split}
\end{align}   
\end{definition}
Equivalently, we have the following closed formula for $\dvp{n}$, for $a\in\N$:
\begin{align*}
\dvp{n} = 
\begin{cases}
\frac{1}{[2a]!}  (t  - [-2a+1] ) (t  - [-2a+3]) \cdots (t  - [2a-3]) (t -[2a-1]), & \text{if } n=2a, \\
\\
\frac{t}{[2a+1]!}  (t  - [-2a+1] ) (t  - [-2a+3]) \cdots (t  - [2a-3]) (t -[2a-1]), &\text{if } n=2a+1.
\end{cases}
\end{align*}
Note the above formulae are formally the same as \eqref{def:dv:evoddK} for $\dv{n} \in {}_\A\Ui_{\rm odd}$ for  $\ka$ an even $q$-integer.

\subsection{Polynomials $\p_n(x)$  and $\p^{(n)}(x)$}
  \label{subsec:pn2}

We define a sequence of polynomials $\p_n(x)$ in a variable $x$, for $n\in \N$,  by letting
\begin{align}
  \label{eq:pn2}
\p_{n+1} =x \p_n  + q^{2-2n} [n] [n-2] \p_{n-1}, \qquad \p_0=1.
\end{align}
Note $\p_n$ is a monic polynomial in $x$ of degree $n$. Also set $\p_n=0$ for $n<0$. These polynomials $\p_n$ will appear in the expansion formula for the $\imath$-divided powers in $\U$.

\begin{example}
Here are the polynomials $\p_n(x)$, for $1\le n \le 5$:
\begin{align*}
& \p_1(x)  =x,
\qquad
\p_2(x) =x^2-1,
\qquad
\p_3(x) =x^3-x,
\\
\p_4(x) &=x^4 +(q^{-4}[3]-1)x^2 -q^{-4}[3],  
\quad
\p_5(x) = x (x^2 - 1) (x^2+q^{-5}[3]! + q^{-6}[5]).
\end{align*}
\end{example}
A simple induction using the recursive formula \eqref{eq:pn2} shows that $\p_n/(x-1) \in \N[q^{-1}][x]$, for $n\ge 2$. In particular,  for $\ka$ any positive odd $q$-integer, we always have $\p_n (\ka) \in \N[q,q^{-1}]$. 

Introduce the monic polynomials $\mathfrak g_n(x)$ of degree $n$:
\begin{align}   \label{eq:gn-ood}
\mathfrak g_n(x) =
\begin{cases}
\prod_{i=1}^{m} (x^2-[2i-1]^2), & \text{ if } n=2m \text{ is even};
\\
x \prod_{i=1}^{m} (x^2-[2i-1]^2), & \text{ if } n=2m+1 \text{ is odd}.
\end{cases}
\end{align}

Define 
\[
\p^{(n)}(x) =\p_n(x)/[n]!, 
\qquad
\mathfrak g^{(n)}(x) =\mathfrak g_n(x)/[n]!. 
\]
Then Equation \eqref{eq:pn2} implies that
\begin{align}
  \label{eq:pn2div}
x \cdot \p^{(n)}  =[n+1] \p^{(n+1)}-q^{2-2n} [n-2] \p^{(n-1)} \quad (n\ge 1).
\end{align}

Our next goal is to prove that $\p^{(n)}(\ka)$ are integral, i.e., $\p^{(n)}([2\ell-1]) \in \A$, for all $n, \ell \in \Z$; see Proposition~\ref{prop:fg:odd}. This is achieved by relating $\p^{(n)}$ to $\mathfrak g^{(n)}$ for varied $n$. 

\begin{lem} \label{lem:odd:g-integral}
For $\ka =[2\ell-1]$ as in \eqref{eq:oddK}, we have $\mathfrak g^{(n)} (\ka) \in \A$, for all $n\in \N$. 
\end{lem}

\begin{proof}
We separate the cases for $n=2m+1$ and $n= 2m$. 
Noting that
\begin{equation}  \label{eq:A2}
\ka -[2i-1] =[2\ell-1]-[2i-1] = [\ell-i] (q^{\ell+i-2} +q^{-\ell-i+2}),
\end{equation} 
we have
 \begin{align*}
 \g^{(2m)}(\ka) &= \frac1{[2m]!} \prod_{i=1-m}^{m} (\ka-[2i-1])
 =  
 \qbinom{\ell+m-1}{2m} \prod_{i=1-m}^{m}(q^{\ell+i-2} +q^{-\ell-i+2} ) \in \A.
 \end{align*}
 
Similarly, using \eqref{eq:A2} we have
 \begin{align*}
 \g^{(2m+1)}(\ka) &= \frac1{[2m+1]!} \ka \prod_{i=1-m}^{m} (\ka-[2i-1])
 \\
 &= \frac1{[2m+1]!} \prod_{i=-m}^{m} (\ka-[2i-1]) - \frac1{[2m]!}  \prod_{i=1-m}^{m} (\ka-[2i-1])
 \\
 &=\qbinom{\ell+m}{2m+1} \prod_{i=-m}^{m}(q^{\ell+i-2} +q^{-\ell-i+2} ) 
  + \qbinom{\ell+m-1}{2m} \prod_{i=1-m}^{m}(q^{\ell+i-2} +q^{-\ell-i+2} ) \in \A.
  \end{align*}
  The lemma is proved. 
\end{proof}

\begin{prop}
 \label{prop:fg:odd}
For $m \in \N$, we have
\begin{align*}
\p^{(2m)} &= \sum_{a=0}^{m-1} q^{(3-2m)a} \qbinom{m-1}{a}_{q^2} \mathfrak g^{(2m-2a)},
\\
\p^{(2m+1)} &= \sum_{a=0}^{m-1} q^{(1-2m)a} \qbinom{m-1}{a}_{q^2} \mathfrak g^{(2m-2a+1)}.
\end{align*}
In particular, we have $\p^{(n)} (\ka) \in \A$, for all $n\in \N$ and all $\ka =[2\ell-1]$ with $\ell \in \Z$. 
\end{prop}

\begin{proof}
 It follows from  these formulae for $\p^{(n)}$ and Lemma~\ref{lem:odd:g-integral} that $\p^{(n)} ([2\ell-1]) \in \A$.

Let us prove these formulae. Note that 
\begin{align}  \label{recursive:g2}
\begin{split}
x \cdot \g^{(2a-1)} &=[2a] \g^{(2a)} +   [2a-1] \g^{(2a-2)},
\\
x \cdot \g^{(2a)} &=  [2a+1] \g^{(2a+1)}, \quad \text{ for } a\ge 1.
\end{split}
\end{align}

We prove by induction on $m$, with the base cases for $\p^{(1)}$ and $\p^{(2)}$ being clear. We separate in two cases. 

(1) Let us prove the formula for $\p^{(2m+1)}$, assuming the formulae for $\p^{(2m-1)}$ and $\p^{(2m)}$.
By \eqref{eq:pn2div} and \eqref{recursive:g2} we have 
\begin{align*} 
[2m+1] \p^{(2m+1)} 
&= x \p^{(2m)} +q^{2-4m} [2m-2] \p^{(2m-1)}  \\
&= \sum_{a=0}^{m-1} q^{(3-2m)a} \qbinom{m-1}{a}_{q^2} x \cdot \g^{(2m-2a)}
\\&\qquad 
  +\sum_{a=0}^{m-2} q^{(3-2m)a+2-4m}  [2m-2] \qbinom{m-2}{a}_{q^2} \mathfrak g^{(2m-2a-1)}
  \\
&= \sum_{a=0}^{m-1} q^{(3-2m)a} [2m-2a+1] \qbinom{m-1}{a}_{q^2}   \g^{(2m-2a+1)}
\\&\qquad 
  +\sum_{a=0}^{m-2} q^{(3-2m)a+2-4m}  [2m-2] \qbinom{m-2}{a}_{q^2} \mathfrak g^{(2m-2a-1)}
  \\
&= \sum_{a=0}^{m-1} q^{(3-2m)a} [2m-2a+1] \qbinom{m-1}{a}_{q^2}   \g^{(2m-2a+1)}
\\&\qquad 
  +\sum_{a=1}^{m-1} q^{(3-2m)(a-1)+2-4m}  [2m-2] \qbinom{m-2}{a-1}_{q^2} \mathfrak g^{(2m-2a+1)}
  \\
  &=[2m+1] \sum_{a=0}^{m-1} q^{(1-2m)a} \qbinom{m-1}{a}_{q^2} \mathfrak g^{(2m-2a+1)},
\end{align*}
where the last equation results from shifting the second summation by $a\to a-1$ and using
the identity 
\[
q^{2a} [2m-2a+1] \qbinom{m-1}{a}_{q^2} +q^{2a-2m-1}  [2m-2] \qbinom{m-2}{a-1}_{q^2}
=[2m+1] \qbinom{m-1}{a}_{q^2}. 
\]
\vspace{3mm} 

(2) We shall prove the formula for $\p^{(2m+2)}$, assuming the formulae for $\p^{(2m+1)}$ and $\p^{(2m)}$.
By \eqref{eq:pn2div} and \eqref{recursive:g2} we have  
\begin{align*} 
&[2m+2] \p^{(2m+2)} 
\\
&= x \p^{(2m+1)} +q^{-4m} [2m-1] \p^{(2m)} 
 \\
&= \sum_{a=0}^{m-1} q^{(1-2m)a} \qbinom{m-1}{a}_{q^2} x\cdot \g^{(2m-2a+1)}
+ \sum_{a=0}^{m-1} q^{(3-2m)a-4m} [2m-1] \qbinom{m-1}{a}_{q^2} \g^{(2m-2a)}
\\%
&= \sum_{a=0}^{m-1} q^{(1-2m)a} \qbinom{m-1}{a}_{q^2} 
\left( [2m-2a+2] \g^{(2m-2a+2)} +[2m-2a+1]  \g^{(2m-2a)} \right)
\\& \qquad 
+ \sum_{a=0}^{m-1} q^{(3-2m)a-4m} [2m-1] \qbinom{m-1}{a}_{q^2} \g^{(2m-2a)}. 
\end{align*}
By combining the like terms for $\g^{(2m-2a)}$ and using the identity
\[
q^{(1-2m)a} [2m-2a+1] + q^{(3-2m)a-4m} [2m-1]
=q^{(1-2m)(a+1)} [4m-2a],
\]
we have
\begin{align*} 
[2m+2] \p^{(2m+2)} 
&= \sum_{a=0}^{m-1} q^{(1-2m)a}  [2m-2a+2] \qbinom{m-1}{a}_{q^2} \g^{(2m-2a+2)}  
\\& \qquad 
+ \sum_{a=0}^{m-1} q^{(1-2m)(a+1)} [4m-2a] \qbinom{m-1}{a}_{q^2} \g^{(2m-2a)}
\\%
&= [2m+2]  \sum_{a=0}^{m} q^{(1-2m)a} \qbinom{m}{a}_{q^2} \mathfrak g^{(2m-2a+2)},
\end{align*}
where the last equation results from shifting the second summation index from $a\to a-1$ and using the identity
\[
[2m-2a+2] \qbinom{m-1}{a}_{q^2} + [4m-2a+2] \qbinom{m-1}{a-1}_{q^2} 
= [2m+2] \qbinom{m}{a}_{q^2}.
\]
The proposition is proved.
\end{proof}

\subsection{Formulae for $\dvp{n}$ with odd $\ka$}

\begin{thm}   
  \label{thm:dvp:oddKappa}
Let $\ka$ be an odd $q$-integer. Then, for $m\ge 0$, we have
\begin{align}
\dvp{2m} &= 
\sum_{b=0}^{2m} \sum_{a=0}^{b} \sum_{c\geq 0} q^{\binom{2c}{2}+2c - b(2m-b-2c) -a(b-a)} 
\p^{(2m-b-2c)}(\kappa) 
\label{t2m:evodd} \\
&\qquad \qquad \qquad\quad \cdot \Y^{(a)}  \qbinom{\kk;1-m}{c} K^{b-2m+2c} F^{(b-a)}, 
\notag
\\
\dvp{2m+1} &= 
\sum_{b=0}^{2m+1} \sum_{a=0}^{b} \sum_{c\geq 0} q^{\binom{2c}{2} -b(2m-b-2c+1) -a(b-a)}
 \p^{(2m-b-2c+1)}(\kappa)
\label{t2m+1:evodd}
\\
&\qquad\qquad\qquad\quad \cdot \Y^{(a)}  \qbinom{\kk;1-m}{c} K^{b-2m+2c-1} F^{(b-a)}. 
\notag
\end{align}
\end{thm}

The proof of Theorem~\ref{thm:dvp:oddKappa} will be given in \S\ref{subsec:proof:dvp:oddKappa} below. 
By applying the anti-involution $\vs$ we convert the formulae in Theorem~\ref{thm:dvp:oddKappa} as follows.

\begin{thm}  
  \label{thm:dvp:oddKappa2}
Let $\ka$ be an odd $q$-integer. Then we have, for $m\ge 1$,  
\begin{align}
\dvp{2m} &= 
\sum_{b=0}^{2m} \sum_{a=0}^{b} \sum_{c\geq 0} 
(-1)^c q^{c +b(2m-b-2c) +a(b-a)}  \p^{(2m-b-2c)}(\kappa) 
 \label{t2m:oddevKa}  \\
&\qquad \qquad \qquad\quad \cdot  F^{(b-a)} K^{b-2m+2c} \qbinom{\kk;m-c}{c}  \Y^{(a)}
\notag
\\
\dvp{2m+1}   &= 
\sum_{b=0}^{2m+1} \sum_{a=0}^{b} \sum_{c\geq 0} 
(-1)^c q^{3c +b(2m-b-2c+1) +a(b-a)} \p^{(2m-b-2c+1)}(\kappa)
 \label{t2m+1:oddevKa} \\
&\qquad\qquad\qquad\quad \cdot F^{(b-a)} K^{b-2m+2c-1} \qbinom{\kk;m-c}{c}  \Y^{(a)}.
\notag
\end{align}
\end{thm}

\begin{proof}
By Lemma~\ref{lem:anti}, the anti-involution $\vs$ sends 
$\qbinom{\kk;1-m}{c}\mapsto (-1)^c q^{2c(c+1)} \qbinom{\kk;m-c}{c}$, 
$q\mapsto q^{-1}$, while fixing $\ka, K, \Y^{(a)}, F^{(a)}$ and $\dvev{n}$. The formulae \eqref{t2m:oddevKa}--\eqref{t2m+1:oddevKa}  now follows from \eqref{t2m:evodd}--\eqref{t2m+1:evodd} in Theorem~\ref{thm:dvp:oddKappa}. 
%
%
\end{proof}

The following corollary is immediate from \eqref{eq:qbinom-hc}, Proposition~\ref{prop:fg:odd}, and Theorem~\ref{thm:dvp:oddKappa}. 
\begin{cor}  
We have $\dvp{n} \in {}_\A \Ui_{\rm{ev}}$, for all $n$. 
\end{cor}

\begin{rem}   
Note $\p_n(1)=0$ for $n\ge 2$, and $\p_0(1)=\p_1(1)=1$. 
The formulae in Theorems~\ref{thm:dvp:oddKappa} and  \ref{thm:dvp:oddKappa2} in the special case for $\ka=1$ recover the formulae in \cite[Theorem~4.1, Proposition~4.3]{BeW18}.
\end{rem}

\begin{example}
The formulae of $\dvp{n}$, for $1\le n\le 3$, in Theorem~\ref{thm:dvp:oddKappa} reads as follows. 
\begin{align*}
\dvp{1} 
&= F+ \Y + K^{-1},
\\
\dvp{2} 
&= b^{(2)}  + q^3 [\kk;0]+\ka( q^{-1}K^{-1}F +q^{-1}\Y K^{-1} ) +  \frac{\ka^{2}-1}{[2]} K^{-2},
\\
\dvp{3} 
&=  b^{(3)} + q \Y [\kk;0] + q [\kk;0] F 
+ \ka( q^{-2} \Y^{(2)} K^{-1} + q^{-3} \Y K^{-1} F + q^{-2} K^{-1} F^{(2)} + q [\kk;0] K^{-1})
\\
 & \qquad + q^{-2}\frac{\ka^{2}-1}{[2]}(\Y K^{-2} + K^{-2} F)  
 + \frac{\ka^3 -\ka}{[3]!} K^{-3}.
\end{align*}
\end{example}

\subsection{The $\imath$-canonical basis for $\Udot_{\rm{ev}}$ with odd $\ka$}
  \label{sec:iCB:evwt-oddK}

Recall from \S\ref{sec:iCB:ev} the $\imath$-canonical basis on simple $\U$-modules $L(\mu)$, for $\mu \in \N$. 
  
\begin{thm}  \label{thm:iCB:evwt-oddK}
    $\quad$
\begin{enumerate}
\item
Let $n\in \N$. 
For each integer $\la \gg n$, the element $\dvp{n} \vev$ is an $\imath$-canonical basis element for $L(2\la)$. 
 
 \item
The set $\{\dvp{n} \mid n \in \N \}$ forms the $\imath$-canonical basis for ${\U}^\imath$ (and an $\A$-basis in ${}_\A \Ui_{\rm{ev}}$). 
\end{enumerate}
\end{thm}
 
 \begin{proof}
Let $\la, m \in \N$. Recall $\cbinom{m}{c}$ from \eqref{cbinom} and $\qbinom{\kk; a}{c}$ from \eqref{kbinom}. It follows by a direct computation using Theorem~\ref{thm:dvp:oddKappa2} and \eqref{eq:qbinom-hc} that 
\begin{align}
\dvp{2m} \vev &= 
\sum_{b=0}^{2m} \sum_{c\geq 0} 
(-1)^c q^{c +b(2m-b-2c)}  \p^{(2m-b-2c)}(\kappa) F^{(b)} K^{b-2m+2c} \qbinom{\kk;m-c}{c}  \vev
\notag
\\
&=  \sum_{b=0}^{2m} \sum_{c \geq 0} q^{-2c^2-c +(b-2\la)(2m-b-2c)}  \p^{(2m-b-2c)}(\kappa) \cbinom{m-\la-c}{c}  F^{(b)}  \vev.
\label{t2m:evwt-oF}  
\end{align}
Similarly using Theorem~\ref{thm:dvp:oddKappa2} we have 
\begin{align}
\dvp{2m+1} \vev &=  
\sum_{b=0}^{2m+1} \sum_{c\geq 0} 
(-1)^c q^{3c +b(2m-b-2c+1)} \p^{(2m-b-2c+1)}(\kappa)
 F^{(b)} K^{b-2m+2c-1} \qbinom{\kk;m-c}{c} \vev
\notag
\\
&=  \sum_{b=0}^{2m+1} \sum_{c \geq 0} q^{-2c^2+c +(b-2\la)(2m-b-2c+1)} \p^{(2m-b-2c+1)}(\kappa) \cbinom{m-\la-c}{c} F^{(b)} \vev. 
\label{t2m+1:evwt-oF} 
\end{align}

By a similar argument for \eqref{eq:lattice}, using \eqref{t2m:evwt-oF}--\eqref{t2m+1:evwt-oF}   we obtain
 \begin{align*}
 \dvp{n} \vev  & \in F^{(n)}  \vev + \sum_{b<n}q^{-1} \Z[q^{-1}] F^{(b)} \vev,
 \qquad \text{ for } \la \gg n.
 \end{align*}
 
The second statement follows now from the definition of the $\imath$-canonical basis for $\Ui$ using
the projective system $\{ L(2\la) \}_{\la\ge 0}$; cf. \cite[\S6]{BW18b}. 
 \end{proof}

\subsection{Proof of Theorem~\ref{thm:dvp:oddKappa} }
  \label{subsec:proof:dvp:oddKappa}
We prove the formulae for $\dvp{n}$ by induction on $n$, in two steps (1)--(2) below. The base cases when $n=1,2$ are clear.

(1) We shall prove the formula \eqref{t2m+1:evodd} for $\dvp{2m+1}$, assuming the formula \eqref{t2m:evodd} for $\dvp{2m}$.

Recall $[2m+1] \dvp{2m+1} = t \cdot \dvp{2m}$, and $t= F + \Y  +\ka K^{-1}$. 
Let us compute 
\[
I=\Y  \dvp{2m},
\qquad
II=F \dvp{2m},
\qquad
III=\ka K^{-1} \dvp{2m}.
\]

First we have
\begin{align*}
I 
&= \sum_{b=0}^{2m} \sum_{a=0}^{b} \sum_{c\geq 0} 
q^{\binom{2c}{2}+2c - b(2m-b-2c) -a(b-a)} 
[a+1] \p^{(2m-b-2c)}(\kappa) 
 \\
&\qquad \qquad \qquad\quad \cdot \Y^{(a+1)}  \qbinom{\kk;1-m}{c} K^{b-2m+2c} F^{(b-a)}, 
\\%
&= \sum_{b=0}^{2m+1} \sum_{a=0}^{b} \sum_{c\geq 0} 
q^{\binom{2c}{2}+2c - (b-1)(2m-b-2c+1) -(a-1)(b-a)} 
[a] \p^{(2m-b-2c+1)}(\kappa) 
 \\
&\qquad \qquad \qquad\quad \cdot \Y^{(a)}  \qbinom{\kk;1-m}{c} K^{b-2m+2c-1} F^{(b-a)},
\end{align*}
where the last equation is obtained by shifting indices $a\to a-1$, $b\to b-1$.
Using \eqref{FYn} we  have
\begin{align*}
II  
&=  \sum_{b=0}^{2m} \sum_{a=0}^{b} \sum_{c\geq 0} q^{\binom{2c}{2}+2c - b(2m-b-2c) -a(b-a)} 
\p^{(2m-b-2c)}(\kappa) 
 \\
& \qquad\quad \cdot 
 \left( 
   q^{-2a} \Y^{(a)} F + \Y^{(a-1)} \frac{q^{3-3a}K^{-2}-q^{1-a}}{q^2-1}
  \right )
\qbinom{\kk;1-m}{c} K^{b-2m+2c} F^{(b-a)}
=II^1 +II^2, 
\\
\end{align*}  
where $II^1$ and $II^2$ are the two natural summands associated to the plus sign.
By shifting index $b\to b-1$, we further have
\begin{align*}
II^1
&=\sum_{b=0}^{2m+1} \sum_{a=0}^{b} \sum_{c\geq 0} 
q^{\binom{2c}{2}+2c -(b+1)(2m-b-2c+1) -a(b-a+1)} 
[b-a] \p^{(2m-b-2c+1)}(\kappa) 
 \\
&\qquad \qquad \qquad\quad \cdot 
   \Y^{(a)}  \qbinom{\kk;-m}{c} K^{b-2m+2c-1} F^{(b-a)}.
\end{align*}
By shifting the indices $a\to a+1$, $b\to b+1$ and $c\to c-1$, we further have
\begin{align*}
II^2 
&=\sum_{b=0}^{2m+1} \sum_{a=0}^{b} \sum_{c\geq 0} 
q^{\binom{2c-2}{2}+2c - (b+1)(2m-b-2c+1) -(a+1)(b-a)-2} 
\p^{(2m-b-2c+1)}(\kappa) 
 \\
&\qquad \qquad \qquad\quad \cdot 
 \Y^{(a)} \frac{q^{-3a}K^{-2}-q^{-a}}{q^2-1}
 \qbinom{\kk;1-m}{c-1} K^{b-2m+2c-1} F^{(b-a)}.
\end{align*}  
Using \eqref{eq:pn2div} we also compute
\begin{align*}
III 
&= \sum_{b=0}^{2m} \sum_{a=0}^{b} \sum_{c\geq 0} 
q^{\binom{2c}{2}+2c - b(2m-b-2c) -a(b-a)-2a} 
\ka \cdot \p^{(2m-b-2c)}(\kappa) 
 \\
&\qquad \qquad \qquad\quad \cdot \Y^{(a)}  \qbinom{\kk;1-m}{c} K^{b-2m+2c-1} F^{(b-a)}
\\
&= \sum_{b=0}^{2m+1} \sum_{a=0}^{b} \sum_{c\geq 0} 
q^{\binom{2c}{2}+2c - b(2m-b-2c) -a(b-a)-2a}  [2m-b-2c+1]  
 \\
&\qquad \qquad \qquad\quad \cdot 
\p^{(2m-b-2c+1)}(\kappa) \Y^{(a)}  \qbinom{\kk;1-m}{c} K^{b-2m+2c-1} F^{(b-a)}
\\%
&\quad - \sum_{b=0}^{2m+1} \sum_{a=0}^{b} \sum_{c\geq 0} 
q^{\binom{2c-2}{2}+2c - (b+2) (2m-b-2c+2) -a(b-a)-2a}  [2m-b-2c]  
 \\
&\qquad \qquad \qquad\quad \cdot
\p^{(2m-b-2c+1)}(\kappa)  \Y^{(a)}  \qbinom{\kk;1-m}{c-1} K^{b-2m+2c-3} F^{(b-a)}.
\end{align*}
(Note that we have shifted the index $c\to c-1$ in the last summand above.)

Collecting the formulae for $I, II^1, II^2$ and $III$  gives us
\[
t \cdot \dvp{2m} = 
 \sum_{0 \le a \le b \le 2m+1}  \sum_{c\ge 0}  \p^{(2m-b-2c+1)}(\kappa)  \Y^{(a)} 
 \ms H_{a,b,c}  K^{b-2m+2c-1} F^{(b-a)}, 
\]
where   
\begin{align*}
\ms H_{a,b,c} := 
& \; q^{\binom{2c}{2}+2c - (b-1)(2m-b-2c+1) -(a-1)(b-a)} 
[a]  \qbinom{\kk;1-m}{c}  
\\
&+q^{\binom{2c}{2}+2c -(b+1)(2m-b-2c+1) -a(b-a+1)} 
[b-a] \qbinom{\kk;-m}{c} 
\\
&+q^{\binom{2c-2}{2}+2c - (b+1)(2m-b-2c+1) -(a+1)(b-a)-2} 
 \frac{q^{-3a}K^{-2}-q^{-a}}{q^2-1} \qbinom{\kk;1-m}{c-1}  
\\
&+q^{\binom{2c}{2}+2c - b(2m-b-2c) -a(b-a)-2a}  
[2m-b-2c+1]    \qbinom{\kk;1-m}{c}
\\
&+q^{\binom{2c-2}{2}+2c - (b+2) (2m-b-2c+2) -a(b-a)-2a}  
[2m-b-2c]   \qbinom{\kk;1-m}{c-1} K^{-2}.
\end{align*}
Recall $[2m+1] \dv{2m+1} = t \cdot \dvp{2m}$. To prove the formula \eqref{t2m+1:evodd} for $\dv{2m+1}$,  by the PBW basis theorem and the inductive assumption it suffices to prove the following identity, for all $a,b,c$:
\begin{align}  \label{eq:H4}
\ms H_{a,b,c} = q^{\binom{2c}{2} -b(2m-b-2c+1) -a(b-a)}
 [2m+1] \qbinom{\kk;1-m}{c}.
\end{align}
Thanks to $q^{2m-a+1} [a] -[2m+1] = q^{-a} [a-2m-1]$, we can combine the RHS\eqref{eq:H4} with the first summand of LHS\eqref{eq:H4}.
Hence, after canceling out the $q$-powers $q^{\binom{2c}{2}- b(2m-b-2c+1) -a(b-a) -a}$
on both sides, we see that \eqref{eq:H4} is equivalent to the following identity, for all $a, b, c$:
\begin{equation}  \label{ABCD4}
\ms A+\ms B+\ms C+\ms D_1 +\ms D_2=0,
\end{equation}
where 
\begin{align*}
\ms A &= [a-2m-1] \qbinom{\kk;1-m}{c},
\\
\ms B &= q^{4c-2m+b-1} [b-a]  \qbinom{\kk;-m}{c},   
\\
\ms C &= q^{2a-2m} \frac{q^{-3a}K^{-2}-q^{-a}}{q^2-1} \qbinom{\kk;1-m}{c-1},   
\\
\ms D_1 &= q^{2c+b-a} [2m-b-2c+1]  \qbinom{\kk;1-m}{c},
\\
\ms D_2 &= - q^{2c-4m+b-a-1} [2m-b-2c] \qbinom{\kk;1-m}{c-1} K^{-2}.
\end{align*}

Let us prove the identity \eqref{ABCD4}. Using \eqref{kbinom2}, we can write $\ms B=\ms B_1+\ms B_2$, where 
\begin{align*}
\ms B_1 &=  q^{4c-2m+b-1} [b-a]  \qbinom{\kk;-m}{c},
   \quad
\ms B_2 = -  q^{4c-6m+b-1} [b-a]  \qbinom{\kk;-m}{c}.
\end{align*}
Noting that  
\[
q^{2a-2m}\frac{q^{-3a}K^{-2}-q^{-a}}{q^2-1} = q^{2m-4c-a}\frac{( q^{4c-4m}K^{-2}-1)}{q^2-1}
 +  q^{-2c-1} [2m-2c-a],
 \] 
we rewrite $\ms C=\ms C_1+\ms C_2$, where
\begin{align*}
\ms C_1 &= q^{2m-2c-a-1} [2c] \qbinom{\kk; 1-m}{c},
   \quad
\ms C_2 =   q^{-2c-1} [2m-2c-a]  \qbinom{\kk; 1-m}{c-1}.
\end{align*}

A direct computation gives us
\begin{align*}
\ms A+\ms B_1+\ms C_1+\ms D_1 
&=(1-q^{-2}) [2c] [2m-2c-a]  \qbinom{\kk; 1-m}{c},
\\
\ms C_2 +(\ms B_2+\ms D_2)  
&= \ms C_2 -  q^{2c-4m-1} [2m-2c-a] \qbinom{\kk; 1-m}{c-1} K^{-2}
\\
&= - (1-q^{-2}) [2c] [2m-2c-a]  \qbinom{\kk; 1-m}{c}.
\end{align*}
Summing up these two equations, we have $\ms A+\ms B+\ms C+\ms D_1+\ms D_2=0,$ whence \eqref{ABCD4}, completing Step~(1). 

\vspace{3mm}

(2) Assuming the formulae for $\dvp{n}$ with $n\le 2m+1$, we shall now prove the following formula \eqref{t2m+2:evodd} for $\dvp{2m+2}$ (obtained with $m$ replaced by $m+1$ in \eqref{t2m:evodd}):
\begin{align}
\dvp{2m+2} &= 
\sum_{b=0}^{2m+2}  \sum_{a=0}^{b} \sum_{c \geq 0} q^{\binom{2c}{2} +2c- b(2m-b-2c+2) -a(b-a)} 
\p^{(2m-b-2c+2)}(\kappa) 
  \label{t2m+2:evodd}
 \\
&\qquad \qquad \qquad\quad \cdot \Y^{(a)}  \qbinom{\kk;-m}{c} K^{b-2m+2c-2} F^{(b-a)}. 
\notag
\end{align}
The proof is based on the recursion $t \cdot \dvp{2m+1} =[2m+2] \dvp{2m+2} +[2m+1] \dvp{2m}$. 
Recall $t= F +\Y +\ka K^{-1}$ and $\dvp{2m+1}$ from \eqref{t2m+1:evodd}. We shall compute 
\[
\texttt{I}=\Y  \dvp{2m+1},
\qquad
\texttt{II}=F \dvp{2m+1},
\qquad
\texttt{III} =\ka K^{-1} \dvp{2m+1}, 
\]
respectively. First by by shifting indices $a\to a-1$ and $b\to b-1$, we have
\begin{align*}
\texttt{I} 
&=\sum_{b=0}^{2m+1} \sum_{a=0}^{b} \sum_{c \geq 0} q^{\binom{2c}{2} - b(2m-b-2c+1) -a(b-a)}
[a+1] \p^{(2m-b-2c+1)}(\kappa) 
\\
&\qquad\qquad\qquad\quad \cdot \Y^{(a+1)}  \qbinom{\kk;1-m}{c} K^{b-2m+2c-1} F^{(b-a)}
\\
&=\sum_{b=0}^{2m+2} \sum_{a=0}^{b} \sum_{c \geq 0} 
q^{\binom{2c}{2} - (b-1)(2m-b-2c+2) -(a-1)(b-a)}
[a] \p^{(2m-b-2c+2)}(\kappa) 
\\
&\qquad\qquad\qquad\quad \cdot \Y^{(a)}  \qbinom{\kk;1-m}{c} K^{b-2m+2c-2} F^{(b-a)}.
\end{align*}
Using \eqref{FYn} we   also have
\begin{align*}
\texttt{II} 
&= \sum_{b=0}^{2m+1} \sum_{a=0}^{b} \sum_{c \geq 0} q^{\binom{2c}{2} - b(2m-b-2c+1) -a(b-a)}
 \p^{(2m-b-2c+1)}(\kappa) 
\\
&\quad \cdot 
 \left( 
   q^{-2a} \Y^{(a)} F + \Y^{(a-1)} \frac{q^{3-3a}K^{-2}-q^{1-a}}{q^2-1}
  \right )
 \qbinom{\kk;1-m}{c} K^{b-2m+2c-1} F^{(b-a)}
=\texttt{II}_1 +\texttt{II}_2,
\end{align*}
where $\texttt{II}_1$ and $\texttt{II}_2$ are the two natural summands associated to the plus sign.
By shifting the index $b\to b-1$, we further have
\begin{align*}
\texttt{II}_1 
&= \sum_{b=0}^{2m+2} \sum_{a=0}^{b} \sum_{c \geq 0} 
q^{\binom{2c}{2} - (b+1)(2m-b-2c+2) -a(b-a-1)-2a}
[b-a] \p^{(2m-b-2c+2)}(\kappa) 
\\
&\qquad\qquad\qquad\quad \cdot 
    \Y^{(a)} 
    \qbinom{\kk;-m}{c} K^{b-2m+2c-2} F^{(b-a)}. 
\end{align*}
By shifting the indices $a\to a+1$, $b\to b+1$ and $c\to c-1$, we further have
\begin{align*}
\texttt{II}_2 
&= \sum_{b=0}^{2m+2} \sum_{a=0}^{b} \sum_{c \geq 0} 
q^{\binom{2c-2}{2} - (b+1)(2m-b-2c+2) -(a+1)(b-a)}
 \p^{(2m-b-2c+2)}(\kappa) 
\\
&\qquad\qquad\qquad\quad \cdot 
 \Y^{(a)} \frac{q^{-3a}K^{-2}-q^{-a}}{q^2-1}
 \qbinom{\kk;1-m}{c-1} K^{b-2m+2c-2} F^{(b-a)}.
 \end{align*}

Using \eqref{eq:pn2div} we also compute
\begin{align*}
\texttt{III}  
&= \sum_{b=0}^{2m+1} \sum_{a=0}^{b} \sum_{c \geq 0} q^{\binom{2c}{2} - b(2m-b-2c+1) -a(b-a)-2a}
\ka \cdot \p^{(2m-b-2c+1)}(\kappa) 
\\
&\qquad\qquad\qquad\quad \cdot \Y^{(a)}  \qbinom{\kk;1-m}{c} K^{b-2m+2c-2} F^{(b-a)}
\\
&= \sum_{b=0}^{2m+2} \sum_{a=0}^{b} \sum_{c \geq 0} q^{\binom{2c}{2} - b(2m-b-2c+1) -a(b-a)-2a}
[2m-b-2c+2]  
\\
&\qquad\qquad\qquad\quad \cdot 
\p^{(2m-b-2c+2)}(\kappa) \Y^{(a)}  \qbinom{\kk;1-m}{c} K^{b-2m+2c-2} F^{(b-a)}
\\%
& - \sum_{b=0}^{2m+2} \sum_{a=0}^{b} \sum_{c \geq 0} 
q^{\binom{2c-2}{2} - (b+2)(2m-b-2c+3) -a(b-a)-2a+2}
[2m-b-2c+1] 
\\
&\qquad\qquad\qquad\quad \cdot 
\p^{(2m-b-2c+2)}(\kappa) \Y^{(a)}  \qbinom{\kk;1-m}{c-1} K^{b-2m+2c-4} F^{(b-a)}.
\end{align*}
(Note that we have shifted the index $c\to c-1$ in the last summand above.)

Collecting the formulae for $\texttt{I}, \texttt{II}_1, \texttt{II}_2$, and $\texttt{III}$, we obtain 
\[
t \cdot \dvp{2m+1} = 
 \sum_{0 \le a \le b \le 2m+2}  \sum_{c\ge 0} \p^{(2m-b-2c+2)}(\kappa)  
  \Y^{(a)} \ms L_{a,b,c} K^{b-2m+2c-2} F^{(b-a)},
\]
where  
\begin{align*}
\ms L_{a,b,c} := 
& \; q^{\binom{2c}{2} - (b-1)(2m-b-2c+2) -(a-1)(b-a)} 
[a] \qbinom{\kk;1-m}{c} 
\\
&+q^{\binom{2c}{2} - (b+1)(2m-b-2c+2) -a(b-a-1)-2a} 
[b-a]  \qbinom{\kk;-m}{c}  
\\
&+q^{\binom{2c-2}{2} - (b+1)(2m-b-2c+2) -(a+1)(b-a)}
\frac{q^{-3a}K^{-2}-q^{-a}}{q^2-1} \qbinom{\kk;1-m}{c-1}  
\\
&+q^{\binom{2c}{2} - b(2m-b-2c+1) -a(b-a)-2a}
[2m-b-2c+2]  \qbinom{\kk;1-m}{c}
\\
&- q^{\binom{2c-2}{2} - (b+2)(2m-b-2c+3) -a(b-a)-2a+2}
[2m-b-2c+1]  \qbinom{\kk;1-m}{c-1} K^{-2}. 
\end{align*}

On the other hand, using \eqref{t2m:evodd} (with an index shift $c\to c-1$) and \eqref{t2m+2:evodd} we write 
\begin{align*}
 [2m+2] & \dvp{2m+2} +[2m+1] \dvp{2m} 
\\
&=\sum_{0 \le a \le b \le 2m+2} \sum_{c \geq 0} 
\p^{(2m-b-2c+2)}(\kappa) \Y^{(a)} \ms R_{a,b,c} K^{b-2m+2c-2} F^{(b-a)}, 
\end{align*}
where  
\begin{align*}
\ms R_{a,b,c} 
&:= q^{\binom{2c}{2} +2c- b(2m-b-2c+2) -a(b-a)} 
[2m+2]  \qbinom{\kk;-m}{c}  
\\
&\qquad + q^{\binom{2c-2}{2} +2c- b(2m-b-2c+2) -a(b-a)-2} 
[2m+1]  \qbinom{\kk;1-m}{c-1}.    
\end{align*}

To prove the formula \eqref{t2m+2:oddev} for $\dvp{2m+2}$, it suffices to show that, for all $a,b,c$,
\begin{equation}
  \label{L=R4}
\ms L_{a,b,c}= \ms R_{a,b,c}.
\end{equation}
Canceling the $q$-powers $q^{\binom{2c}{2} - (b-1)(2m-b-2c+2) -(a-1)(b-a)}$ on both sides, we see that the identity \eqref{L=R4} is equivalent to the following identity, for all $a,b,c$:
\begin{align}
 \label{l=r4}
 \begin{split}
[a] \qbinom{\kk;1-m}{c}  
&+q^{4c-4m+b-4}  [b-a]  \qbinom{\kk;-m}{c}   
\\
&+q^{2a-4m-1} \frac{q^{-3a}K^{-2}-q^{-a}}{q^2-1} \qbinom{\kk;1-m}{c-1}   
\\
&+q^{2c-2m+b-a-2} [2m-b-2c+2]  \qbinom{\kk;1-m}{c}
\\
&- q^{2c-6m+b-a-3} [2m-b-2c+1]  \qbinom{\kk;1-m}{c-1} K^{-2}
\\
&= q^{4c-2m+a-2}  [2m+2]  \qbinom{\kk;-m}{c}  
+ q^{a-2m-1}  [2m+1]  \qbinom{\kk;1-m}{c-1}.  
   \end{split}
\end{align}

By combining the second summand of LHS with the first summand of RHS as well as combining the third summand of LHS with the second summand of RHS, the identity \eqref{l=r4} is reduced to the following equivalent identity, for all $a, b, c$:
\begin{equation}  \label{WXYZ4}
\ms W+\ms X+\ms Y+\ms Z_1+\ms Z_2=0,
\end{equation}
where  
\begin{align*}
\ms W & =[a] \qbinom{\kk;1-m}{c},
\\
\ms X & =q^{4c-2m+b-2}  [b-a-2m-2]  \qbinom{\kk;-m}{c},   
\\
\ms Y & =  \frac{q^{-a-4m-1}K^{-2}-q^{a+1}}{q^2-1} \qbinom{\kk;1-m}{c-1},    
\\
\ms Z_1 & =q^{2c-2m+b-a-2} [2m-b-2c+2]  \qbinom{\kk;1-m}{c},
\\
\ms Z_2 & = - q^{2c-6m+b-a-3} [2m-b-2c+1]  \qbinom{\kk;1-m}{c-1} K^{-2}.
\end{align*}

Let us finally prove the identity \eqref{WXYZ4}. Using \eqref{kbinom2}, we can write $\ms X=\ms X_1+\ms X_2$, where 
\begin{align*}
\ms X_1 &=q^{4c-2m+b-2}  [b-a-2m-2]  \qbinom{\kk; 1-m}{c},
   \\
\ms X_2 &= - q^{4c-6m+b-2}  [b-a-2m-2]  \qbinom{\kk; 1-m}{c-1} K^{-2}.
\end{align*}
Noting that  
\[
\frac{q^{-a-4m-1}K^{-2}-q^{a+1}}{q^2-1}  
= q^{-4c-a-1}\frac{( q^{4c-4m}K^{-2}-1)}{q^2-1}
 +  q^{-2c-1} [-2c-a-1],
 \] 
we rewrite $\ms Y=\ms Y_1+\ms Y_2$, where
\begin{align*}
\ms Y_1 &=q^{-2c-a-2} [2c] \qbinom{\kk; 1-m}{c},
   \qquad
\ms Y_2 = -q^{-2c-1} [2c+a+1] \qbinom{\kk; 1-m}{c-1}.
\end{align*}
A direct computation shows that
\begin{align*}
\ms W+\ms X_1+\ms Y_1+\ms Z_1 &=
- (1-q^{-2}) [2c+a+1] [2c] \qbinom{\kk; 1-m}{c},
\\
(\ms X_2+\ms Z_2) +\ms Y_2  
&= q^{2c-4m-1} [2c+a+1] \qbinom{\kk; 1-m}{c-1} K^{-2} +\ms Y_2
\\
&= (1-q^{-2}) [2c+a+1] [2c] \qbinom{\kk; 1-m}{c}.
\end{align*}
Summing up these two equations, we obtain $\ms W+\ms X+\ms Y+\ms Z_1+\ms Z_2=0$, whence \eqref{WXYZ4}, completing Step ~(2).

The proof of Theorem~\ref{thm:dvp:oddKappa} is completed. \qed

\section{The $\imath$-divided powers $\dvd{n}$ for odd weights and odd $\ka$}
  \label{sec:oddoddK}
  
In this section~\ref{sec:oddoddK} we always take $\ka$ to be an odd $q$-integer, i.e., 
\begin{equation*}
\ka =[2\ell-1], \qquad  \text{ for } \ell  \in \Z.
\end{equation*} 

\subsection{Definition of $\dvd{n}$ for odd $\ka$}

\begin{definition}
Set $\dvd{1} =t= F +\Y +\ka K^{-1}$. 
The divided powers $\dvd{n}$, for $n\ge 1$, are defined by the recursive relations: 
\begin{align}   
  \label{eq:tt:evoddKa}
\begin{split}
t \cdot \dvd{2a-1} &=[2a] \dvd{2a},
\\
t \cdot \dvd{2a} &=  [2a+1] \dvd{2a+1} +   [2a] \dvd{2a-1}, \quad \text{ for } a\ge 1.
\end{split}
\end{align}   
\end{definition}
Equivalently, we have the following closed formula for $\dvd{n}$:
\begin{align}
\label{def:dvd:evoddKa}
\dvd{n} = 
\begin{cases}
\frac{t}{[2a]!} (t  - [-2a+2])(t -[-2a+4]) \cdots (t -[2a-4]) (t  - [2a-2]), & \text{if } n=2a, \\
\\
\frac{1}{[2a+1]!} (t  - [-2a])(t -[-2a+2]) \cdots (t -[2a-2]) (t  - [2a]), &\text{if } n=2a+1.
\end{cases}
\end{align}
Note the formulae for $\dvd{n}$ with odd $\ka$ is formally identical to the formulae for $\dvev{n}$ with even $\ka$. 

\subsection{Formulae for $\dvd{n}$ with odd $\ka$}

Recall the polynomials $\p^{(n)}$, for $n\ge 0$, from \S\ref{subsec:pn2}.

\begin{thm}  
   \label{thm:dvd:oddKappa}
Assume $\ka$ is an odd $q$-integer. Then we have, for $m\ge 1$,   
\begin{align}
\dvd{2m} &= 
\sum_{b=0}^{2m} \sum_{a=0}^{b} \sum_{c \geq 0} q^{\binom{2c}{2}-2c -b(2m-b-2c)-a(b-a)} 
\p^{(2m-b-2c)}(\kappa)
\label{t2moddodd} \\
&\qquad \qquad \qquad\quad \cdot \Y^{(a)}  \LR{\kk;2-m}{c} K^{b-2m+2c} F^{(b-a)}, 
\notag
\\
\dvd{2m-1} &= 
\sum_{b=0}^{2m-1} \sum_{a=0}^{b} \sum_{c \geq 0} q^{\binom{2c}{2} - b(2m-b-2c-1) -a(b-a)}
 \p^{(2m-b-2c-1)}(\kappa)
\label{t2m-1oddodd}
\\
&\qquad\qquad\qquad\quad \cdot \Y^{(a)}  \LR{\kk;2-m}{c} K^{b-2m+2c+1} F^{(b-a)}. 
\notag
\end{align}
\end{thm}

The proof of Theorem~\ref{thm:dvd:oddKappa} will be given in \S\ref{subsec:proof:dvd:oddKappa} below. 
By applying the anti-involution $\vs$ we convert the formulae in Theorem~\ref{thm:dvd:oddKappa} as follows.

\begin{thm}   
 \label{thm:dvd:oddKappa2}
Assume $\ka$ is an odd $q$-integer. Then we have, for $m\ge 1$,   
\begin{align}
\dvd{2m} 
&= \sum_{b=0}^{2m} \sum_{a=0}^{b} \sum_{c \geq 0} 
(-1)^c  q^{c +b(2m-b-2c) +a(b-a)}  \p^{(2m-b-2c)}(\kappa) 
\label{t2moddodd2} \\
&\qquad\qquad\qquad\quad \cdot F^{(b-a)} K^{b-2m+2c} \LR{\kk;m-c}{c}  \Y^{(a)}, 
\notag
\\%
\dvd{2m-1} &= 
\sum_{b=0}^{2m-1} \sum_{a=0}^{b} \sum_{c \geq 0} 
(-1)^c q^{-c+ b(2m-b-2c-1) +a(b-a)}
 \p^{(2m-b-2c-1)}(\kappa)
\label{t2m-1oddodd2} \\
&\qquad\qquad\qquad\quad \cdot F^{(b-a)} K^{b-2m+2c+1}   \LR{\kk;m-c}{c}  \Y^{(a)}.
\notag
\end{align}
\end{thm}

\begin{proof}
Recall from Lemma~\ref{lem:anti2} that the anti-involution $\vs$ on $\U$   
fixes $F,  \Y,  K,  \dvd{n}, \ka$ while sending
$q \mapsto q^{-1}$, 
$\LR{\kk;2-m}{c}\mapsto (-1)^c q^{2c(c-1)}  \LR{\kk;m-c}{c}$.
The formulae \eqref{t2moddodd2}--\eqref{t2m-1oddodd2} now follow by applying $\vs$ to the formulae \eqref{t2moddodd}--\eqref{t2m-1oddodd} in Theorem~\ref{thm:dvd:oddKappa}. 
%
%
\end{proof}

The following corollary is immediate from \eqref{eq:LRhc}, Proposition~\ref{prop:fg:odd}, and Theorem~\ref{thm:dvd:oddKappa}. 

\begin{cor}  
We have $\dvd{n} \in {}_\A \Ui_{\rm{odd}}$, for all $n$. 
\end{cor}

\begin{rem}   
Note $\p_n(1)=0$ for $n\ge 2$, and $\p_0(1)=\p_1(1)=1$. 
The formulae in Theorems~\ref{thm:dvd:oddKappa} and  \ref{thm:dvd:oddKappa2} in the special case for $\ka=1$ recover the formulae in \cite[Theorem~5.1, Proposition~5.3]{BeW18}.
\end{rem}

\begin{example}
The formulae of $\dvd{n}$, for $1\le n\le 3$, in Theorem~\ref{thm:dvd:oddKappa} read as follows. 
\begin{align*}
\dvd{1} 
&= F +\Y + \ka K^{-1},
\\
\dvd{2} 
&= \Y^{(2)} +q^{-1} \Y F + F^{(2)}  + q^{-1} [[\kk;1]] 
 + \ka (q^{-1}K^{-1}F + q^{-1} \Y K^{-1})  +  \frac{\ka^2-1}{[2]} K^{-2},
\\
\dvd{3}  &   = b^{(3)} + q[[\kk;0]]F+ q \Y [[\kk;0]]
\\
&\qquad 
+ \big(q^{-2} \Y^{(2)}K^{-1} +q^{-3} \Y K^{-1}F + q^{-2} K^{-1}F^{(2)} 
+ q [[h;0]] K^{-1} \big) \ka
\\
&\qquad + \frac{(\ka^2-1)}{[2]} (q^{-2}\Y K^{-2} +q^{-2}K^{-2}F) 
+\frac{\ka^3 -\ka}{[3]!} K^{-3}.
\end{align*}
\end{example}

\subsection{The $\imath$-canonical basis for $\Udot_{\text{odd}}$ with odd $\ka$}
  \label{sec:iCB:evwt-oddK}

Recall from \S\ref{sec:iCB:ev} the $\imath$-canonical basis on simple $\U$-modules $L(\mu)$, for $\mu \in \N$. 
  
\begin{thm}  \label{thm:iCB:evwt-oddK}
    $\quad$
\begin{enumerate}
\item
Let $n\in \N$. 
For each integer $\la \gg n$, the element $\dvd{n} \vodd$ is an $\imath$-canonical basis element for $L(2\la+1)$. 
 
 \item
The set $\{\dvd{n} \mid n \in \N \}$ forms the $\imath$-canonical basis for ${\U}^\imath$ (and an $\A$-basis in ${}_\A \Ui_{\rm{odd}}$). 
\end{enumerate}
\end{thm}
 
 \begin{proof}
 Let $\la, m \in \N$. Recall $\cbinom{m}{c}$ from \eqref{cbinom}. It follows by a direct computation using Theorem~\ref{thm:dvd:oddKappa2} and \eqref{eq:LRhc} that 
\begin{align}
\dvd{2m} \vodd &= 
\sum_{b=0}^{2m} \sum_{c \geq 0} (-1)^c  q^{c +b(2m-b-2c)}  \p^{(2m-b-2c)}(\kappa) 
F^{(b)} K^{b-2m+2c} \LR{\kk;m-c}{c}  \vodd
\notag
\\
&=  \sum_{b=0}^{2m} \sum_{c \geq 0} q^{-2c^2+c +(b-2\la)(2m-b-2c)}  \p^{(2m-b-2c)}(\kappa) \cbinom{m-\la-c}{c}  F^{(b)}  \vodd.
\label{t2m:oddoF}  
\end{align}
Similarly using Theorem~\ref{thm:dvd:oddKappa2} we have 
\begin{align}
\dvd{2m-1} \vodd &=  
\sum_{b=0}^{2m-1} \sum_{c \geq 0} 
(-1)^c q^{-c+ b(2m-b-2c-1)} \p^{(2m-b-2c-1)}(\kappa) F^{(b)} K^{b-2m+2c+1}   \LR{\kk;m-c}{c} \vodd
\notag
\\
&=  \sum_{b=0}^{2m-1} \sum_{c \geq 0} q^{-2c^2-c +(b-2\la)(2m-b-2c+1)} \p^{(2m-b-2c+1)}(\kappa) \cbinom{m-\la-c}{c} F^{(b)} \vodd. 
\label{t2m-1:oddoF} 
\end{align}

By a similar argument for \eqref{eq:lattice}, using \eqref{t2m:oddoF}--\eqref{t2m-1:oddoF}  we obtain
 \begin{align*}
 \dvd{n} \vodd  & \in F^{(n)}  \vodd + \sum_{b<n}q^{-1} \Z[q^{-1}] F^{(b)} \vodd,
 \qquad \text{ for } \la \gg n.
 \end{align*}
 
The second statement follows now from the definition of the $\imath$-canonical basis for $\Ui$ using
the projective system $\{ L(2\la+1) \}_{\la\ge 0}$; cf. \cite[\S6]{BW18b}. 
 \end{proof}

\subsection{Proof of Theorem~\ref{thm:dvd:oddKappa} }
  \label{subsec:proof:dvd:oddKappa}
  
We prove the formulae for $\dvd{n}$ by induction on $n$, in two separate cases (1)--(2) below. The base cases when $n=1,2$ are clear.

(1) We shall prove the formula \eqref{t2moddodd} for $\dvd{2m}$, assuming the formula \eqref{t2m-1oddodd} for $\dvd{2m-1}$.

Recall $[2m] \dvd{2m} = t \cdot \dvd{2m-1}$, and $t= F +\Y +\ka K^{-1}$. 
Let us compute 
\[
I=\Y \dvd{2m-1},
\qquad
II=F \dvd{2m-1},
\qquad
III=\ka K^{-1} \dvd{2m-1}.
\]

We compute
\begin{align*}
I 
&= 
\sum_{b=0}^{2m-1} \sum_{a=0}^{b} \sum_{c \geq 0} 
q^{\binom{2c}{2} - b(2m-b-2c-1) -a(b-a)} [a+1] \p^{(2m-b-2c-1)}(\kappa)
\\
&\qquad\qquad\qquad\quad \cdot \Y^{(a+1)}  \LR{\kk;2-m}{c} K^{b-2m+2c+1} F^{(b-a)}
\\
 &= 
\sum_{b=0}^{2m} \sum_{a=0}^{b} \sum_{c \geq 0} 
q^{\binom{2c}{2} - (b-1)(2m-b-2c) -(a-1)(b-a)} [a] \p^{(2m-b-2c)}(\kappa)
\\
&\qquad\qquad\qquad\quad \cdot \Y^{(a)}  \LR{\kk;2-m}{c} K^{b-2m+2c} F^{(b-a)}. 
\end{align*}
where the last equation is obtained by shifting indices $a\to a-1$, $b\to b-1$.
Using \eqref{FYn} we  have
\begin{align*}
II  
&=\sum_{b=0}^{2m-1} \sum_{a=0}^{b} \sum_{c \geq 0} q^{\binom{2c}{2} - b(2m-b-2c-1) -a(b-a)}\p^{(2m-b-2c-1)}(\kappa) \\
  & \qquad \cdot 
  \left( 
   q^{-2a} \Y^{(a)} F + \Y^{(a-1)} \frac{q^{3-3a}K^{-2}-q^{1-a}}{q^2-1}
  \right )
    \LR{\kk;2-m}{c} K^{b-2m+2c+1} F^{(b-a)}
   =II^1 +II^2,
\end{align*}
where $II^1$ and $II^2$ are the two natural summands associated to the plus sign.
By shifting the index $b\to b-1$ and then adding some zero terms, we obtain
\begin{align*}
II^1
&=\sum_{b=0}^{2m} \sum_{a=0}^{b} \sum_{c \geq 0} q^{\binom{2c}{2} - (b+1)(2m-b-2c) -a(b-a+1)} [b-a] \p^{(2m-b-2c)}(\kappa)  \\
&\qquad\qquad\qquad \cdot \Y^{(a)}  \LR{\kk;1-m}{c} K^{b-2m+2c} F^{(b-a)}.
\end{align*}
By shifting the indices $a\to a+1$, $b\to b+1$ and then adding some zero terms, we also obtain
\begin{align*}
II^2 
&=\sum_{b=0}^{2m} \sum_{a=0}^{b} \sum_{c \geq 0} q^{\binom{2c-2}{2} - (b+1)(2m-b-2c) -(a+1)(b-a)}\p^{(2m-b-2c)}(\kappa) \\
  & \qquad\qquad\qquad \cdot 
 \Y^{(a)} \frac{q^{-3a}K^{-2}-q^{-a}}{q^2-1}
    \LR{\kk;2-m}{c-1} K^{b-2m+2c} F^{(b-a)}.
\end{align*}
By the identity \eqref{eq:pn2div} we have
\begin{align*}
III 
&= \sum_{b=0}^{2m-1} \sum_{a=0}^{b} \sum_{c \geq 0} q^{\binom{2c}{2} - b(2m-b-2c-1) -a(b-a)-2a}
 \ka \cdot \p^{(2m-b-2c-1)}(\kappa)
\\ 
&\qquad\qquad\qquad\quad \cdot \Y^{(a)}  \LR{\kk;2-m}{c} K^{b-2m+2c} F^{(b-a)}
\\
&= \sum_{b=0}^{2m} \sum_{a=0}^{b} \sum_{c \geq 0} q^{\binom{2c}{2} - b(2m-b-2c-1) -a(b-a)-2a}
[2m-b-2c]  
\\
&\qquad\qquad\qquad\quad \cdot \p^{(2m-b-2c)} \Y^{(a)}  \LR{\kk;2-m}{c} K^{b-2m+2c} F^{(b-a)}
\\
&\quad - \sum_{b=0}^{2m} \sum_{a=0}^{b} \sum_{c \geq 0} 
 q^{\binom{2c-2}{2} -(b+2)(2m-b-2c+1) -a(b-a)-2a +2}
[2m-b-2c-1]  
\\
&\qquad\qquad\qquad\quad \cdot 
\p^{(2m-b-2c)} \Y^{(a)}  \LR{\kk;2-m}{c-1} K^{b-2m+2c-2} F^{(b-a)}.
\end{align*}
(Note that we have shifted the index $c\to c-1$ in the last summand above.)

Collecting the formulae for $I, II^1, II^2, III$ gives us
\[
t \cdot \dvd{2m-1} = \sum_{0 \le a \le b \le 2m}  \sum_{c\ge 0} 
\p^{(2m-b-2c)}(\kappa) \Y^{(a)} \mf H_{a,b,c}  K^{b-2m+2c} F^{(b-a)}, 
\]
where 
\begin{align*}
\mf H_{a,b,c} := 
& \; q^{\binom{2c}{2} - (b-1)(2m-b-2c) -(a-1)(b-a)} [a]  \LR{\kk;2-m}{c}  
\\
&+ q^{\binom{2c}{2} - (b+1)(2m-b-2c) -a(b-a+1)} [b-a] \LR{\kk;1-m}{c}  
\\
&+ q^{\binom{2c-2}{2} - (b+1)(2m-b-2c) -(a+1)(b-a)} %
   \frac{q^{-3a}K^{-2}-q^{-a}}{q^2-1} \LR{\kk;2-m}{c-1}   
\\%
&+ q^{\binom{2c}{2} - b(2m-b-2c-1) -a(b-a)-2a}[2m-b-2c] \LR{\kk;2-m}{c}  
\\
&- q^{\binom{2c-2}{2} -(b+2)(2m-b-2c+1) -a(b-a)-2a +2}
[2m-b-2c-1] \LR{\kk;2-m}{c-1} K^{-2},  
\end{align*}

Recall $[2m] \dvd{2m} = t \cdot \dvd{2m-1}$. To prove the formula \eqref{t2moddodd} for $\dvd{2m}$,  by the PBW basis theorem and the inductive assumption it suffices to prove the following identity, for all $a,b,c$:
\begin{align}  \label{eq:Hoo}
\mathfrak H_{a,b,c} = q^{\binom{2c}{2}-2c -b(2m-b-2c)-a(b-a)} [2m] \LR{\kk;2-m}{c}.
\end{align}  
Thanks to $q^{2m-a} [a] -[2m] =q^{-a} [a-2m]$, we can combine the RHS\eqref{eq:Hoo} with the first summand of the LHS\eqref{eq:Hoo}.
Hence, after canceling out the $q$-powers $q^{\binom{2c}{2}-2c -b(2m-b-2c)-a(b-a)-a}$ on both sides, we see that \eqref{eq:Hoo} is equivalent to the following identity, for all $a, b, c$:
\begin{equation}  \label{ABCD2}
\mathfrak A+\mathfrak B+\mathfrak C+\mathfrak D_1 +\mathfrak D_2=0,
\end{equation}
where  
\begin{align*}  
\mathfrak A &=  [a-2m] \LR{\kk;2-m}{c},
  \\
\mathfrak B &=  q^{b+4c-2m} [b-a] \LR{\kk;1-m}{c},
\\
\mathfrak C  &=  q^{2a-2m+3} \frac{q^{-3a}K^{-2}-q^{-a}}{q^2-1} \LR{\kk;2-m}{c-1},
\\
\mathfrak D_1 &=  q^{2c+b-a} [2m-b-2c] \LR{\kk;2-m}{c},
   \\
\mathfrak D_2 &= - q^{2c-4m+b-a+3} [2m-b-2c-1] \LR{\kk;2-m}{c-1} K^{-2}.
\end{align*}

Let us prove the identity \eqref{ABCD2}. Using \eqref{eq:commLR}, 
we can write $\mathfrak B=\mathfrak B_1+\mathfrak B_2$, where 
\begin{align*}
\mathfrak B_1 &=  q^{b+4c-2m} [b-a] \LR{\kk;2-m}{c},
   \quad
\mathfrak B_2 = - q^{b+4c-6m+4} [b-a] \LR{\kk;2-m}{c-1} K^{-2}.
\end{align*}
 Noting that  
\[
\frac{q^{-3a}K^{-2}-q^{-a}}{q^2-1} = q^{4m-4c-3a-4}\frac{( q^{4c-4m+4}K^{-2}-q^2)}{q^2-1}
 +  q^{2m-2c-2a-2} [2m-2c-a-1],
 \] 
we rewrite $\mf C=\mf C_1+\mf C_2$, where 
\begin{align*}
\mf C_1 &= q^{2m-a-2c-2} [2c] \LR{\kk;2-m}{c},
   \qquad
\mf C_2 =   q^{1-2c} [2m-2c-a-1] \LR{\kk;2-m}{c-1}.
\end{align*}

A direct computation gives us  
\begin{align*}
\mf A+\mf B_1+\mf C_1+\mf D_1 
&=  (1-q^{-2}) [2c] [2m-2c-a-1] p^{(2m-b-2c)}(\kappa)   \LR{\kk;2-m}{c},
\\
\mf C_2 +(\mf B_2+\mf D_2)  
&= \mf C_2 -  q^{2c-4m+3} [2m-2c-a-1]  p^{(2m-b-2c)}(\kappa)   \LR{\kk;2-m}{c-1} K^{-2}
\\
&=  - (1-q^{-2}) [2c] [2m-2c-a-1] p^{(2m-b-2c)}(\kappa)   \LR{\kk;2-m}{c}.
\end{align*}
Summing up these two equations, we have $\mf A+\mf B+\mf C+\mf D_1+\mf D_2=0,$ whence \eqref{ABCD2}, completing Step~(1). 

\vspace{3mm}

(2) Assuming the formulae for $\dvd{n}$ with $n\le 2m$, we shall now prove the following formula for $\dvd{2m+1}$ (obtained from \eqref{t2m-1oddodd} with $m$ replaced by $m+1$):
\begin{align}
\dvd{2m+1} &= 
\sum_{b=0}^{2m+1} \sum_{a=0}^{b} \sum_{c \geq 0} 
q^{\binom{2c}{2} -b(2m-b-2c+1) -a(b-a)} p^{(2m-b-2c+1)}(\kappa) 
\label{t2m+1oddodd}
\\
&\qquad\qquad\qquad\quad \cdot  \Y^{(a)}  \LR{\kk;1-m}{c} K^{b-2m+2c-1} F^{(b-a)}. 
\notag
\end{align}
We first need to compute $t \cdot \dvd{2m}$, where we recall $t= F +\Y  +\ka K^{-1}$ and $\dvd{2m}$ from \eqref{t2moddodd}. We shall compute 
\[
\texttt I=\Y  \dvd{2m},
\qquad
\texttt{II}=F \dvd{2m},
\qquad
\texttt{III} =\ka K^{-1} \dvd{2m}, 
\]
respectively. First we have
\begin{align*}
\texttt{I} 
&=\sum_{b=0}^{2m} \sum_{a=0}^{b} \sum_{c \geq 0} q^{\binom{2c}{2}-2c -b(2m-b-2c)-a(b-a)} 
[a+1] \p^{(2m-b-2c)}(\kappa)
  \\
&\qquad \qquad \qquad\quad \cdot \Y^{(a+1)}  \LR{\kk;2-m}{c} K^{b-2m+2c} F^{(b-a)}
\\ %
&=\sum_{b=0}^{2m+1} \sum_{a=0}^{b} \sum_{c \geq 0} 
q^{\binom{2c}{2}-2c -(b-1)(2m-b-2c+1)-(a-1)(b-a)}  [a] \p^{(2m-b-2c+1)}(\kappa)
  \\
&\qquad \qquad \qquad\quad \cdot \Y^{(a)}  \LR{\kk;2-m}{c} K^{b-2m+2c-1} F^{(b-a)},
\end{align*}
where the last equation is obtained by shifting the indices $a\to a-1$, $b\to b-1$. We also have
\begin{align*}
\texttt{II}  
&= \sum_{b=0}^{2m} \sum_{a=0}^{b} \sum_{c \geq 0} 
q^{\binom{2c}{2}-2c -b(2m-b-2c)-a(b-a)} \p^{(2m-b-2c)}(\kappa)
  \\
&\qquad    \cdot 
\left( 
   q^{-2a} \Y^{(a)} F + \Y^{(a-1)} \frac{q^{3-3a}K^{-2}-q^{1-a}}{q^2-1}
  \right )
 \LR{\kk;2-m}{c} K^{b-2m+2c} F^{(b-a)}
=\texttt{II}_1 +\texttt{II}_2,
\end{align*}
where $\texttt{II}_1$ and $\texttt{II}_2$ are the two natural summands associated to the plus sign.
By shifting the index $b\to b-1$, we further have
\begin{align*}
\texttt{II}_1
&=\sum_{b=0}^{2m+1} \sum_{a=0}^{b} \sum_{c \geq 0} 
q^{\binom{2c}{2}-2c -(b+1)(2m-b-2c+1) -a(b-a+1)} [b-a] \p^{(2m-b-2c+1)}(\kappa)
\\
&\qquad \qquad \qquad\quad \cdot
 \Y^{(a)}  \LR{\kk;1-m}{c} K^{b-2m+2c-1} F^{(b-a)}. 
\end{align*}
By shifting the indices $a\to a+1$, $b\to b+1$ and $c\to c-1$, we further have
\begin{align*}
\texttt{II}_2
&=\sum_{b=0}^{2m+1} \sum_{a=0}^{b} \sum_{c \geq 0} 
q^{\binom{2c-2}{2}-2c -(b+1)(2m-b-2c+1)-(a+1)(b-a)+2} \p^{(2m-b-2c+1)}(\kappa)
  \\
&\qquad    \cdot  \Y^{(a)} \frac{q^{-3a}K^{-2}-q^{-a}}{q^2-1}
 \LR{\kk;2-m}{c-1} K^{b-2m+2c-1} F^{(b-a)}.
\end{align*}
Using the identity \eqref{eq:pn2div} we also compute
\begin{align*}
 \texttt{III} =
& \sum_{b=0}^{2m} \sum_{a=0}^{b} \sum_{c \geq 0} 
q^{\binom{2c}{2}-2c -b(2m-b-2c)-a(b-a)-2a} \ka \cdot \p^{(2m-b-2c)}(\kappa)
  \\
& \qquad\quad \cdot \Y^{(a)}  \LR{\kk;2-m}{c} K^{b-2m+2c-1} F^{(b-a)}
  \\%
= & \sum_{b=0}^{2m+1} \sum_{a=0}^{b} \sum_{c \geq 0} 
q^{\binom{2c}{2}-2c -b(2m-b-2c)-a(b-a)-2a} [2m-b-2c+1]  
  \\
& \qquad\quad \cdot \p^{(2m-b-2c+1)}(\kappa) \Y^{(a)}  \LR{\kk;2-m}{c} K^{b-2m+2c-1} F^{(b-a)}
\\
& -  \sum_{b=0}^{2m+1} \sum_{a=0}^{b} \sum_{c \geq 0} q^{\binom{2c-2}{2}-2c -(b+2)(2m-b-2c+2)-a(b-a)-2a+4} 
 [2m-b-2c]  
  \\
& \qquad\quad \cdot \p^{(2m-b-2c+1)}(\kappa) \Y^{(a)}  \LR{\kk;2-m}{c-1} K^{b-2m+2c-3} F^{(b-a)}.
\end{align*}
(Note that we have shifted the index $c\to c-1$ in the last summand above.)

Collecting the formulae for $\texttt{I}, \texttt{II}_1, \texttt{II}_2, \texttt{III}$, we obtain an expression of the form
\[
t \cdot \dvd{2m} = 
 \sum_{0 \le a \le b \le 2m+1}  \sum_{c\ge 0}  
 \p^{(2m-b-2c+1)}(\kappa) \Y^{(a)} \mf L_{a,b,c} K^{b-2m+2c-1} F^{(b-a)},
\]
where
\begin{align*}
\mf L_{a,b,c} :=& \;
q^{\binom{2c}{2}-2c -(b-1)(2m-b-2c+1)-(a-1)(b-a)}  [a]  \LR{\kk;2-m}{c}
\\
&+q^{\binom{2c}{2}-2c -(b+1)(2m-b-2c+1) -a(b-a+1)} [b-a] \LR{\kk;1-m}{c}
\\
&+q^{\binom{2c-2}{2}-2c -(b+1)(2m-b-2c+1)-(a+1)(b-a)+2}  
 \frac{q^{-3a}K^{-2}-q^{-a}}{q^2-1} \LR{\kk;2-m}{c-1}
\\
&+q^{\binom{2c}{2}-2c -b(2m-b-2c)-a(b-a)-2a} [2m-b-2c+1] \LR{\kk;2-m}{c}
\\
&-q^{\binom{2c-2}{2}-2c -(b+2)(2m-b-2c+2)-a(b-a)-2a+4} 
 [2m-b-2c] \LR{\kk;2-m}{c-1} K^{-2}.
\end{align*}
On the other hand, using \eqref{t2m-1oddodd} (with an index shift $c\to c-1$) and \eqref{t2m+1oddodd} we write 
\begin{align*}
[2m+1] & \dvd{2m+1} +[2m] \dvd{2m-1} 
\\
& =\sum_{0 \le a \le b \le 2m+1} \sum_{c \geq 0} 
\p^{(2m-b-2c+1)}(\kappa) \Y^{(a)} \mf R_{a,b,c} K^{b-2m+2c-1} F^{(b-a)}, 
\end{align*}
where  
\begin{align*}
\mf R_{a,b,c} :=&
q^{\binom{2c}{2} -b(2m-b-2c+1) -a(b-a)} [2m+1]  \LR{\kk;1-m}{c}  
\\
&\quad + q^{\binom{2c-2}{2} - b(2m-b-2c+1) -a(b-a)} [2m]  \LR{\kk;2-m}{c-1}.
\end{align*}

To prove the formula \eqref{t2m+1oddodd} for $\dvd{2m+1}$, it suffices to show that, for all $a,b,c$, 
\begin{equation}
  \label{L=R2}
 \mf L_{a,b,c}=\mf R_{a,b,c}.
 \end{equation}
Canceling the $q$-powers $q^{\binom{2c}{2}-2c -(b-1)(2m-b-2c+1)-(a-1)(b-a)}$  on both sides, we see that the identity \eqref{L=R2} is equivalent to the following identity, for all $a,b,c$:
\begin{align}
  \label{l=r2}
 \begin{split}
[a]  \LR{\kk;2-m}{c} 
&+q^{4c-4m+b-2} [b-a] \LR{\kk;1-m}{c}
\\
&+q^{2a-4m+3} \frac{q^{-3a}K^{-2}-q^{-a}}{q^2-1} \LR{\kk;2-m}{c-1}
\\
&+q^{2c-2m+b-a-1} [2m-b-2c+1] \LR{\kk;2-m}{c}
\\
&-q^{2c-6m+b-a+2} [2m-b-2c] \LR{\kk;2-m}{c-1} K^{-2}
 \\
 =& \; q^{4c-2m+a-1} [2m+1]  \LR{\kk;1-m}{c}  
+ q^{a-2m+2} [2m]   \LR{\kk;2-m}{c-1}.
 \end{split}
\end{align}

By combining the second summand of the LHS\eqref{l=r2} with the first summand of the RHS\eqref{l=r2} as well as combining the third summand of the LHS with the second summand of the RHS, the identity \eqref{l=r2} is reduced to the following equivalent identity, for all $a, b, c$:
\begin{equation}  \label{WXYZ2}
\mf W+ \mf X+\mf Y+\mf Z_1+\mf Z_2=0,
\end{equation}
where 
\begin{align*}
\mf W&= [a]  \LR{\kk;2-m}{c},
\\
\mf X&= q^{4c-2m+b-1} [b-a-2m-1] \LR{\kk;1-m}{c},
\\
\mf Y&=  \frac{q^{3-a-4m}K^{-2}-q^{3+a}}{q^2-1} \LR{\kk;2-m}{c-1},  
\\
\mf Z_1&= q^{2c-2m+b-a-1} [2m-b-2c+1] \LR{\kk;2-m}{c},
\\
\mf Z_2&= -q^{2c-6m+b-a+2} [2m-b-2c] \LR{\kk;2-m}{c-1} K^{-2}.
\end{align*}
Let us finally prove the identity \eqref{WXYZ2}. Using \eqref{eq:commLR}, we can write $\mf X=\mf X_1+\mf X_2$, where 
\begin{align*}
\mf X_1 &=q^{4c-2m+b-1} [b-a-2m-1] \LR{\kk; 2-m}{c},
   \\
\mf X_2 &= - q^{4c-6m+b+3} [b-a-2m-1] \LR{\kk; 2-m}{c-1} K^{-2}.
\end{align*}
Noting that  
\[
\frac{q^{3-a-4m}K^{-2}-q^{3+a}}{q^2-1}
= q^{-4c-a-1}\frac{( q^{4c-4m+4}K^{-2}-q^2)}{q^2-1}
 +  q^{1-2c} [-2c-a-1],
 \] 
we rewrite $\mf Y=\mf Y_1+\mf Y_2$, where
\begin{align*}
\mf Y_1 &=q^{-2c-a-2} [2c] \LR{\kk; 2-m}{c},
   \qquad
\mf Y_2 = -q^{1-2c} [2c+a+1] \LR{\kk; 2-m}{c-1}.
\end{align*}
A direct computation shows that
\begin{align*}
\mf W+\mf X_1+\mf Y_1+\mf Z_1 &=
-(1-q^{-2}) [2c+a+1] [2c]  \LR{\kk; 2-m}{c},
\\
(\mf X_2+\mf Z_2) +\mf Y_2 
&= q^{-2m-a+1} [2c+a] \LR{\kk; 2-m}{c-1} K^{-2} +\mf Y_2
\\
&= (1-q^{-2}) [2c+a+1] [2c]  \LR{\kk; 2-m}{c}.
\end{align*}
Summing up these two equations, we obtain $\mf W+\mf X+\mf Y+\mf Z_1+\mf Z_2=0$, whence \eqref{WXYZ2}, completing Step ~(2).

The proof of Theorem~\ref{thm:dvd:oddKappa} is completed. \qed

\appendix
\section{On the $\imath$-divided powers  for generic $\ka$ }
  \label{sec:genK}

In this appendix we provide closed formulae for the second $\imath$-divided power for $\Ui$ with an arbitrary parameter $\overline{\ka} =\ka\in \A$. 

For $\overline{\ka} =\ka\in \A$, we can write 
\[
\ka =\sum_{i\in\Z} c_i q^i,
\quad \text{ where } c_i=c_{-i}, \forall i.
\] 
Define 
\begin{equation}
  \label{eq:mm}
\m= \m(\ka) = \sum_{i\in \Z}(-1)^i c_{2i}.
\end{equation}

\begin{lem}
 \label{lem:A}
Let $\overline{\ka} =\ka\in \A$. We have $\frac{1}{[2]} (\ka - \m) \in \A$. 
\end{lem}

\begin{proof}
Note that
$q^{2i+1} +q^{-2i-1} \in [2] \A$ and $q^{2i} +q^{-2i} -2(-1)^i \in [2] \A$ for all $i\in \Z$. The lemma follows. 
\end{proof}

Recall from \eqref{eq:t} that $t = F +\Y + \ka K^{-1}$.  We shall denote the second $\imath$-divided power in ${}_\A\Ui_{\rm ev}$ (and respectively, ${}_\A\Ui_{\rm odd}$) by $\dvp{2}$ (and respectively, $\dvd{2}$), which is by definition the $\imath$-canonical basis element with a leading term $F^{(2)}$ (see \cite{BW18b}). The following formula was obtained with help from Huanchen Bao. 

\begin{prop}
 \label{prop:tt}
Let $\overline{\ka} =\ka\in \A$.  For $\la \in \Z$, we have 
\[
\dvp{2}   =\frac{t^2- \m^2}{[2]}, 
\qquad 
\dvd{2}  =\frac{t^2-1 + \m^2}{[2]}.    
\]
\end{prop}

\begin{proof}
A direct computation shows that, for $\mu \in \Z$, 
\begin{align}
  \label{eq:t22}
t^2 /[2] \cdot v^+_\mu 
&= F^{(2)} v^+_\mu +q^{1-\mu} \ka F v^+_\mu +\frac{ \frac{1-q^{-2\mu}}{1-q^{-2}} +q^{-2\mu} \ka^2}{q+q^{-1}} v^+_\mu.
\end{align}
Hence we obtain by using \eqref{eq:t22} that 
\begin{align*}
(t^2-\m^2)/[2] \cdot v^+_{2\la} 
&= F^{(2)} v^+_{2\la} +q^{1-2\la} \ka F v^+_{2\la} +\frac{q^{1-2\la} [2\la] +q^{-4\la} \ka^2 -\m^2}{q+q^{-1}} v^+_{2\la}
\\
&= F^{(2)} v^+_{2\la} +q^{1-2\la} \ka F v^+_{2\la} +\frac{q^{1-2\la} [2\la] +q^{-4\la} (\ka^2 -\m^2)
 + (q^{-4\la} -1) \m^2}{q+q^{-1}} v^+_{2\la}.
\end{align*}
Note that (each of the 3 summands of) the last fraction above lies in $\A$ thanks to Lemma~\ref{lem:A}. 
Hence $(t^2-\m^2)/[2]$ is integral, i.e., it lies in ${}_\A\Ui_{\rm ev}$. 
Moreover, the lower terms above all have coefficients in $q^{-1} \Z[q^{-1}]$ for $\la \gg 0$. Hence it follows by definition \cite{BW18b} that $(t^2-\m^2)/[2]$ is an $\imath$-canonical basis element in ${}_\A\Ui_{\rm ev}$. 

On the other hand, by using \eqref{eq:t22} again we have
\begin{align*}
(t^2-1 +\m^2)/[2] \cdot v^+_{2\la+1} 
&= F^{(2)} v^+_{2\la+1} +q^{1-2\la} \ka F v^+_{2\la+1} 
+\frac{q^{-2\la-1} [2\la] +q^{-4\la-2} \ka^2 +\m^2}{q+q^{-1}} v^+_{2\la+1}
\\
&= F^{(2)} v^+_{2\la+1} +q^{1-2\la} \ka F v^+_{2\la+1} 
\\
&\qquad +\frac{q^{-2\la-1} [2\la] +q^{-4\la-2} (\ka^2 -\m^2)
 + (q^{-4\la-2} +1) \m^2}{q+q^{-1}} v^+_{2\la+1}.
\end{align*}
Note that (each of the 3 summands of) the fraction above lies in $\A$ thanks to Lemma~\ref{lem:A}. 
Hence $(t^2-1+ \m^2)/[2] \one_{2\la}$ is integral, i.e., it lies in ${}_\A\Ui_{\rm ev}$. 
Moreover, the lower terms above all have coefficients in $q^{-1} \Z[q^{-1}]$ for $\la \gg 0$. Hence it follows by definition \cite{BW18b} that $(t^2-1+\m^2)/[2]$ is an $\imath$-canonical basis element in ${}_\A\Ui_{\rm odd}$. 
\end{proof}

Recall $[\kk;0] =\frac{K^{-2}-1}{q^4-1}$, and $\llbracket\kk;0\rrbracket =\frac{K^{-2}-q^2}{q^4-1}$.
We leave the verification of the following formulae using Proposition~\ref{prop:tt} to the reader. 
\begin{prop}
 \label{prop:t22}
Let $\overline{\ka} =\ka\in \A$.  The following formulae for $\dvp{2}$ and $\dvd{2}$ in $\U$ hold:
\begin{align*}
\dvp{2} 
&= \Y^{(2)} +q^{-1} \Y F + F^{(2)} +q^{-1}\ka K^{-1}F +q^{-1} \ka \Y K^{-1}  
\\
& \qquad\quad 
+ \big( q + (q^3-q)\ka^{2} \big) \frac{K^{-2}-1}{q^4-1}  
+\frac{\ka^2 -\m^2}{[2]},
\\
\dvd{2} 
&= \Y^{(2)} +q^{-1} \Y F + F^{(2)} +q^{-1}\ka K^{-1}F +q^{-1} \ka \Y K^{-1}  
\\
& \qquad\quad 
+ \big( q + (q^3-q)\ka^{2} \big) \frac{K^{-2}-q^2}{q^4-1}  
+\frac{q^2(\ka^2 -\m^2)}{[2]} +q \m^2. 
\end{align*}
\end{prop}

\begin{example}
Assume $\ka$ is a $q$-integer, i.e., $\ka =[n]$ for some $n\in \Z$. Then $\m$ in \eqref{eq:mm} becomes 
\begin{align*}
\m = \begin{cases}
0, & \text{ for } n \text{ even}, 
\\
1, & \text{ for } n\equiv 1\pmod 4,
\\
-1, & \text{ for } n\equiv -1\pmod 4. 
\end{cases}
\end{align*}
In this case, the formulae in Propositions~\ref{prop:tt} and \ref{prop:t22} reduce to those obtained in the previous sections. 
\end{example}

It is an open question to find closed formula for general $\imath$-divided powers of $t$ with an arbitrary bar invariant parameter $\ka\in \A$. 

It is also an open question to find closed formulae for the $\imath$-canonical bases for the simple $\U$-modules $L(\la)$, for $\la \in \N$, even for $\ka$ being an arbitrary $q$-integer as in Section~\ref{sec:evevK}-\ref{sec:oddoddK}. The explicit fromulas for such $\imath$-canonical bases for $L(\la)$ were known in \cite{BeW18}  for $\ka=0$ or $1$; in this case, they coincide with the nonzero images of the $\imath$-divided powers acting on  $v^+_\la$. 



\end{document}